\documentclass[11pt,reqno]{amsart} % use larger type; default would be 10pt

\usepackage[utf8]{inputenc} % set input encoding (not needed with XeLaTeX)
\usepackage{bbm}

\usepackage[margin=1in]{geometry} % to change the page dimensions
\usepackage{graphicx} % support the \includegraphics command and options

 \usepackage[parfill]{parskip} % Activate to begin paragraphs with an empty line rather than an indent
 
\usepackage{booktabs} % for much better looking tables
\usepackage{array} % for better arrays (eg matrices) in maths
\usepackage{paralist} % very flexible & customisable lists (eg. enumerate/itemize, etc.)
\usepackage{verbatim} % adds environment for commenting out blocks of text & for better verbatim
\usepackage{subfig} % make it possible to include more than one captioned figure/table in a single float
\usepackage{mathrsfs}
\usepackage{amssymb}
\usepackage{amsthm}
\usepackage{amsmath,amsfonts,amssymb,esint}
\usepackage{graphics,color}
\usepackage{enumerate}
\usepackage{mathtools,centernot}
\usepackage{cases}

\usepackage{xfrac}
\newtheorem{theorem}{Theorem}[section]
\newtheorem{lemma}[theorem]{Lemma}
\newtheorem{corollary}[theorem]{Corollary}

\newtheorem{proposition}[theorem]{Proposition}
\newtheorem{remark}[theorem]{Remark}

\newtheorem*{remark 1}{Remark}

\numberwithin{equation}{section} 

\newcommand{\gammalaplace}{(-\Delta)^\gamma}
\newcommand{\norm}[1]{\left\|#1\right\|}
\newcommand{\abs}[1]{\left|#1\right|}

\DeclarePairedDelimiter{\ceil}{\lceil}{\rceil}

\newcommand{\T}{\ensuremath{\mathbb{T}}}
\newcommand*{\R}{\ensuremath{\mathbb{R}}}
\renewcommand*{\S}{\ensuremath{\mathcal{S}}}

\newcommand*{\N}{\ensuremath{\mathbb{N}}}
\newcommand*{\Z}{\ensuremath{\mathbb{Z}}}
\newcommand*{\C}{\ensuremath{\mathbb{C}}}
\newcommand{\eps}{\varepsilon}
\newcommand*{\tr}{\ensuremath{\mathrm{tr\,}}}

\newcommand*{\Id}{\ensuremath{\mathrm{Id}}}

\newcommand*{\RR}{\ensuremath{\mathcal{R}}}
\newcommand*{\RRc}{\ensuremath{\mathcal{B}}}

\usepackage{color, graphicx}
\usepackage{mathrsfs, dsfont}

\usepackage{hyperref}

\def\div{\mathop{\rm div}\nolimits}    %divergence
\def\supp{\mathop{\rm supp}\nolimits}    %support
\def\curl{\mathop{\rm curl}\nolimits}    %curl
\title{Infinitely many Leray-Hopf solutions for the fractional Navier-Stokes equations}

\author{Luigi De Rosa}
\address{EPFL SB, Station 8, 
CH-1015 Lausanne, Switzerland}
\email{luigi.derosa@epfl.ch}

\date{\today}
\begin{document}

\begin{abstract}
We prove the ill-posedness for the Leray-Hopf weak solutions  of the incompressible and ipodissipative Navier-Stokes equations, when the power of the diffusive term $\gammalaplace$ is $\gamma < \sfrac{1}{3}$. We construct infinitely many solutions, starting from the same initial datum, which belong to  $C_{x,t}^{\sfrac{1}{3}-}$ and strictly dissipate their energy in small time intervals. The proof exploits the "convex integration scheme" introduced by C. De Lellis and  L. Sz\'ekelyhidi for the incompressible Euler equations, joining these ideas with new stability estimates for a class of non-local advection-diffusion equations and a local (in time) well-posedness result for the fractional Navier-Stokes system. Moreover we show the existence of dissipative H\"older continuous solutions of Euler equations that can be obtained as a vanishing viscosity limit of Leray-Hopf weak solutions of a suitable fractional Navier-Stokes equations.
\end{abstract}

\maketitle

%\tableofcontents

\section{Introduction}

In this paper we consider the Cauchy problem for the incompressible fractional Navier-Stokes equations 
\begin{equation}\label{NSgamma}
\left\{\begin{array}{l}
\partial_t v+ v\cdot\nabla v +\nabla p +\gammalaplace v=0\\ \\
\div v = 0 \\ \\
v(\cdot,0)= \overline{v},
\end{array}\right.
\end{equation}
in the spatial periodic setting $\T^3=\R^3\setminus \Z^3$,
where $v$ is a vector field representing the velocity of the fluid, $p$ is the hydrodynamic pressure, $\overline{v}:\T^3\rightarrow \R^3$ is any given solenoidal initial data and $\gamma \in (0,\sfrac{1}{3})$. The operator $\gammalaplace$ is the (non-local) diffusive operator, whose Fourier series is given by
\[
\gammalaplace v (x) := \sum_{k \in \Z^3} |k|^{2\gamma} \hat{v}_k e^{ik\cdot x}. 
\]
We are interested in Leray-Hopf weak solutions of \eqref{NSgamma}, namely solutions $v \in L^\infty(\R^+,L^2(\T^3))\cap L^2(\R^+,H^\gamma(\T^3)) $ satisfying \eqref{NSgamma} in the distributional sense, namely such that 

\[
\int_0^\infty \int_{\T^3} [(v \cdot \partial_t \varphi  - v \cdot \gammalaplace \varphi  +  v\otimes v:D \varphi ] (x,s) \, dx\, ds
= - \int_{\T^3} \overline{v} (x)\cdot \varphi (x,0)\, dx\,,
\]
for every smooth test vector field $\varphi \in C^\infty_c (\T^3\times \R, \R^3)$ with $\div \varphi = 0$ (note that $p$ can be recovered uniquely as a distribution if we impose that $\int p \, dx = 0$), and obeying to the global energy inequality
\begin{equation}\label{energyineq}
\frac{1}{2} \int_{\T^3} |v|^2 (x,t)\, dx + \int_s^t\int_{\T^3} |(-\Delta)^{\sfrac{\gamma}{2}} v|^2 (x, \tau)\, dx\, d\tau \leq \frac{1}{2} \int_{\T^3} |v|^2 (x,s)\, dx\, , \qquad \forall\, 0\leq s< t\,.
\end{equation}
As for the Navier-Stokes equations (i.e. the case $\gamma=1$), it is known that such solutions exist. Indeed we have (for the proof, see Theorem 1.1 in \cite{CDLDR2017})

\begin{theorem}\label{lerayexist}
For any $\overline v \in L^2 (\T^3)$ with $\div \overline v=0$ and every $\gamma \in ]0,1[$ there exists a Leray-Hopf weak solution of \eqref{NSgamma}.
\end{theorem}
It is also known that, if the power $\gamma$ of the Laplacian is suitably small, then these solutions are not unique. Indeed in \cite{CDLDR2017} the authors proved the ill-posedness in the case $\gamma <\sfrac{1}{5}$. The question about uniqueness is still open if $\gamma\geq \sfrac{1}{5}$. In this work we partially answer this question, proving the non-uniqueness of such solutions in the range $0<\gamma <\sfrac{1}{3}$. More precisely the main result of this paper is the following 

\begin{theorem}\label{t:locale}
Let $\gamma < \sfrac{1}{3}$. Then there are initial data $\overline v\in L^2 (\T^3)$ with $\div \overline v=0$ for which there exist infinitely many Leray solutions $v$ of \eqref{NSgamma} in $[0,+\infty)\times \T^3$. More precicely, if $\gamma <\beta< \sfrac{1}{3}$, there are initial data $\overline v\in C^\beta (\T^3)$ with $\div \overline v =0$ and a positive time $T$ such that
\begin{itemize}
\item[(a)] there are infinitely many Leray-Hopf solutions of \eqref{NSgamma} and moreover $v \in C^\beta (\T^3 \times [0,T])$;
\item[(b)] such solutions strictly dissipate the total energy in $[0,T]$, i.e. the function (of time only)
\begin{equation}\label{etot}
e_{tot}(t):=\frac{1}{2} \int_{\T^3} |v|^2 (x,t)\, dx + \int_0^t\int_{\T^3} |(-\Delta)^{\sfrac{\gamma}{2}} v|^2 (x, \tau)\, dx\, d\tau
\end{equation}
is strictly decreasing in $[0,T]$.
\end{itemize}
\end{theorem}

The proof of Theorem \ref{t:locale} is achieved by using the "convex integration methods" introduced by C. De Lellis and L. Sz\'ekelyhidi  for the incompressible Euler equations, in particular the costruction used in \cite{BDLSV2017}, where the authors, thanks to the new ideas introduced by P. Isett in \cite{Is2016}, proved the existence of $C_{x,t}^{\sfrac{1}{3}-}$  solutions of Euler equations with prescribed kinetic energy. This methods can be also used to prove the ill-posedness for the distributional solutions of the Navier-Stokes equations (i.e. $\gamma=1$). Indeed, recently, in \cite{BV2017} T. Buckmaster and V. Vicol proved the existence of infinitely many weak solutions of the Navier-Stokes equations with bounded kinetic energy. The solutions constructed in \cite{BV2017} do not even have finite energy dissipation in the sense of $e_{tot}$, thus they are not of Leray-Hopf type.
These iterative methods have already been used to prove ill-posedness results in contexts of fractional powers of the Laplacian. For istance in \cite{BSV2016} they  produce infinitely many solutions of the SQG equation.

In order to use the argument proposed in \cite{BDLSV2017}, we have to construct exact solutions of Eq. \eqref{NSgamma} in small time intervals. The corresponding stability estimates of such solutions, with respect to the initial data, are also needed. To this aim we prove new stability estimates for classical solutions of non-local advection-diffusion equations.

Following \cite{CDLDR2017} we will see that if the exponent $\gamma$ is not too large (in particular $\gamma<\sfrac{1}{3}$), then the methods used in \cite{BDLSV2017} to produce H\"older continuous solutions to the Euler equations with prescribed kinetic energy can be adapted to equations \eqref{NSgamma}. Then we will be able to produce (different) solutions with different kinetic energy profile, let all of them start from the same initial data and keep under control the dissipative part in the definition of $e_{tot}$ (see \eqref{etot}).

As already did in  \cite{CDLDR2017}, also in this case the methods would give us infinitely many weak solutions bounded in $L^\infty(\R^+,L^2(\T^3))$ in the range $\sfrac{1}{3} \leq \gamma <\sfrac{1}{2}$, but (a priori) without any control on $e_{tot}$. Since there will not be a big improvement with respect to  \cite{CDLDR2017}, we are not exploiting the details of the construction in this range.

In order to avoid confusion, for fractional Navier-Stokes equations with some viscosity $\nu>0$ we mean the system
\begin{equation}\label{NSvisc}
\left\{\begin{array}{l}
\partial_t v+ v\cdot\nabla v +\nabla p +\nu \gammalaplace v=0\\ \\
\div v = 0\,.
\end{array}\right.
\end{equation}
When $\nu=0$ they are known as Euler equations. Using the main iterative proposition (Proposition \ref{p:main}) we are able to show the existence of dissipative solutions of Euler which can be obtained as a vanishing viscosity limit of solutions of \eqref{NSvisc}. The main idea is taken from \cite{BV2017} where the authors proved that H\"older continuous solutions of Euler arise as a strong limit in $C_t^0(L^2)$ (as $\nu \rightarrow 0$) of weak solutions of the classical Navier-Stokes equations.

Again by the restriction $\gamma<\sfrac{1}{3}$, we are able to produce a sequence Leray-Hopf weak solutions of \eqref{NSvisc} converging to a  dissipative solution of Euler, as $\nu \rightarrow 0$. More precicely we prove the following
\begin{theorem}\label{vanishvisc}
Let $\beta'<\sfrac{1}{3}$. There exist dissipative solutions $v \in C^{\beta'}([0,T]\times\T^3)$ of Euler such that, if  $ \,0<\gamma < \beta'$, there exists a sequence $\nu_n \rightarrow 0$ and a sequence $v^{(\nu_n)} $ of Leray-Hopf weak solutions of \eqref{NSvisc} such that $v^{(\nu_n)}\rightarrow
v $ strongly in $ C^0([0,T],C^{\beta''}(\T^3))$ for every $\beta''<\beta'$.
\end{theorem}

Also in this case, if we only want to require that the sequence $v^{(\nu_n)}$ is just a sequence of weak solution of Eq. \eqref{NSvisc}, bounded in $L^\infty([0,T],L^2(\T^3))$, we could also prove that for any $\gamma<\sfrac{1}{2}$ there exists a sequence of solutions of \eqref{NSvisc} converging to any H\"older solution of Euler, as $\nu \rightarrow 0$, but in order to be consistent with the arguments of this work, we will not enter in this details.

\section{Proof of Theorem \ref{t:locale}} 

In order to show Theorem \ref{t:locale} we will prove a slightly more general result about Eq. \eqref{NSgamma}. Indeed, using the inductive scheme proposed in \cite{BDLSV2017}, we are able to prove the following 
\begin{theorem}\label{assegnocinetica}
Let $e:[0,1] \rightarrow \R^+$ with the following properties
\begin{itemize}
\item[$(i)$] $\sfrac{1}{2} \leq e(t) \leq 1, \, \forall t \in[0,1]$;
\item[$(ii)$]$ sup_t \, |e'(t)|\leq K, \, $ for some $K>1$.
\end{itemize}
Then for all $\gamma<\beta<\sfrac{1}{3}$ there exists a couple $(v,p)$, solving 
\begin{equation}\label{NSgamma_senzadato}
\left\{\begin{array}{l}
\partial_t v+ v\cdot\nabla v +\nabla p +\gammalaplace v=0\\ \\
\div v = 0 
\end{array}\right.
\end{equation}
in the sense of distributions, such that $ v \in C^\beta (\T^3 \times [0,1])$ and \footnote{Here $\| \cdot\|_\beta$ denotes the H\"older norm, see next section for precise definition}
\begin{align}
&e(t)=\int_{\T^3} |v|^2(x,t)\, dx \, , \label{condizione1}  \\ 
&  \|v\|_\beta\leq C_\beta K^{\sfrac{4}{9}}\, ,\label{condizione2}
\end{align}
 where $C_\beta$ is a constant depending only on $\beta$. Moreover, given any two energy profiles $e_1$ and $e_2$ such that $e_1(0)=e_2(0)$, then the two corresponding solutions $v^{(1)}$ and $v^{(2)}$ start from the same initial data, i.e. $v^{(1)}(\cdot,0)\equiv v^{(2)}(\cdot,0)$.
\end{theorem}
We end this section proving Theorem \ref{t:locale}, then the rest of the paper will be devoted to the proof of Theorem \ref{assegnocinetica}.

\begin{proof}[Proof of Theorem \ref{t:locale}]
Elementary arguments produce for every $K>1$ an  infinite set $\mathscr{E}_K$ of smooth functions $e: [0,1]\to \mathbb R$ with the following properties:
\begin{itemize}
\item[$(i)$] $\sfrac{1}{2} \leq e(t) \leq 1\, , \, $  $\forall t\in[0,1]$;
\item[$(ii)$] $\|e\|_{C^1([0,1])} \leq 2K+2$;
\item[$(iii)$] $e (0) =1$ ;
\item[$(iv)$] $e' (t) \leq - 2K +2\, , $ $\, \forall t \in [0, \frac{1}{4K}]$;
\item[$(v)$] for any pair of distinct elements of $\mathscr{E}_K$ there is a sequence of times converging to
$0$ where they take different values.
\end{itemize}
For each $e \in \mathscr{E}_K$, we now use Theorem \ref{assegnocinetica} to produce infinitely many weak solutions satisfying
\begin{itemize}
\item[$(a)$] $e(t)=\frac{1}{2}\int_{\T^3} |v|^2(x,t)\, dx$ ;
\item[$(b)$] $v \in C^\beta (\T^3 \times [0,1]) , \quad \forall \beta <\sfrac{1}{3}$ ;
\item[$(c)$] $v(\cdot, 0)= \overline{v}$, for some $\overline{v}\in C^{\beta}(\T^3)$ ;
\item[$(d)$] $ \|v\|_\beta\leq  C_\beta K^{\sfrac{4}{9}}$.
\end{itemize}
Let $T=\sfrac{1}{4K}$. We have to show that all these solutions strictly dissipate the total energy, which is equivalent to
\begin{equation}\label{disenegia}
\frac{1}{2}\bigl( e(s)-e(t)\bigl)>\int_s^t \int_{\T^3} |(-\Delta)^{\sfrac{\gamma}{2}} v|^2 (x, \tau)\, dx\, d\tau ,\quad \forall \, 0\leq s <t \leq T\, .
\end{equation}
By our assumptions on the functions $e(t)$ and using Corollary \ref{vlaplv} we have
\begin{align*}
& \frac{1}{2}\bigl( e(s)-e(t)\bigl)\geq (K-1)(t-s) ,\quad \forall 0\leq s <t \leq T;\\
& \int_s^t \int_{\T^3} |(-\Delta)^{\sfrac{\gamma}{2}} v|^2 (x, \tau)\, dx\, d\tau \leq (t-s)C_\varepsilon \|v\|_{\gamma +\varepsilon}^2.
\end{align*}
Chosing $\varepsilon$ so that $\gamma+\varepsilon = \beta$, we see that \eqref{disenegia} holds if the constant $K$ satisfies
\begin{equation}\label{fine}
K -1> C_{\beta,\gamma} K^{\sfrac{8}{9}}\,,
\end{equation}
where $C_{\beta,\gamma}$ depends only on $\gamma$ and $\beta$, but not on $K$.
It is clear that there exists a $K$ (big enough) such that \eqref{fine} is satisfied. Thus we have proved the existence of infinitely many Leray-Hopf solutions in the interval $[0,T]$ satisfying $(a)$ and $(b)$ of Theorem \ref{t:locale}. Finally, using Theorem \ref{lerayexist}, it is not difficoult to show that all these solutions can be prolonged to Leray-Hopf solutions for every $t\geq 0$, thus the proof is concluded.
\end{proof}

\section{Stability estimates for classical solutions of non-local advection-diffusion equations and classical solutions of the fractional Navier-Stokes equations}

In the following $m=0,1,2,\dots$, $\alpha\in (0,1)$, and $\theta$ is a multi-index. We introduce the usual (spatial) 
H\"older norms as follows.
First of all, the supremum norm is denoted by $\|f\|_0:=\sup_{\T^3\times [0,T]}|f|$. We define the H\"older seminorms 
as
\begin{equation*}
\begin{split}
[f]_{m}&=\max_{|\theta|=m}\|D^{\theta}f\|_0\, ,\\
[f]_{m+\alpha} &= \max_{|\theta|=m}\sup_{x\neq y, t}\frac{|D^{\theta}f(x, t)-D^{\theta}f(y, t)|}{|x-y|^{\alpha}}\, ,
\end{split}
\end{equation*}
where $D^\theta$ are {\em space derivatives} only.
The H\"older norms are then given by
\begin{eqnarray*}
\|f\|_{m}&=&\sum_{j=0}^m[f]_j\\
\|f\|_{m+\alpha}&=&\|f\|_m+[f]_{m+\alpha}.
\end{eqnarray*}
Moreover, we will write $[f (t)]_\alpha$ and $\|f (t)\|_\alpha$ when the time $t$ is fixed and the
norms are computed for the restriction of $f$ to the $t$-time slice.
\subsection{Maximum principle and stability estimates}
We begin by stating a maximum principle result for a non-local operator. The proof is standard, since, as for the local case (i.e. using the Laplacian), we have that $\gammalaplace u(x_0)\geq 0$ whenever $x_0$ is a global maximum point of $u$ (see for instance the integral representation formula given in \cite[Theorem 1.5]{RS2016}).
\begin{theorem}[Maximum principle]\label{MP}
Define $ Q_T:= \T^3 \times (0,T]$. Let $L$ be the pseudo-differential operator defined as $Lu = (v \cdot \nabla ) u+\nu \gammalaplace u $, where $u: \T^3 \times[0,T] \rightarrow \R $ , $ v:\T^3 \times [0,T] \rightarrow \R^3$ is a given vector field and $\nu>0$, $0<\gamma\leq 1$. The following holds:
\begin{itemize}
\item[(i)] if $ u_t +Lu \leq 0 $ in $ Q_T$, then $\, \max_{\overline{Q}_T} u = \max_{\T^3 \times\{0\}} u$;
\item[(ii)]if $ u_t +Lu \geq 0 $ in $ Q_T$, then $\,\min_{\overline{Q}_T} u = \min_{\T^3 \times\{0\}} u$.
\end{itemize}
\end{theorem}
In 
Using Theorem \ref{MP} we can prove a stability estimate for a general class of non-local parabolic equations. Indeed we have 

\begin{proposition}\label{C0stability}
Let $u: \T^3 \times[t_0,T] \rightarrow \R^3 $ be a solution of the Cauchy problem
\begin{equation}\label{nonlocale}
\left\{\begin{array}{l}
u_t+Lu = f \qquad \text{in} \quad \T^3\times(t_0,T)\\ \\
u(\cdot,t_0)=u_0 \qquad \text{in} \quad \T^3.
\end{array}\right.
\end{equation}
Then for any $t \in [t_0,T]$ we have 
\begin{equation}\label{stabilityestimate1}
\|u(t)\|_0 \leq \| u_0\|_0 + \int_{t_0}^t \|f(s)\|_0 \,ds\, ,
\end{equation}
\begin{equation}\label{stabilityestimate2}
[u(t)]_1 \leq [u_0]_1e^{(t-t_0)[v]_1} + \int_{t_0}^te^{(t-s)[v]_1} [f(s)]_1 \,ds \, ,
\end{equation}
and, more generally, for any $N\geq 2$ there exists a constant $C=C_N$ so that 
\begin{align}\label{stabilityestimate3}
[u(t)]_N &\leq \, \bigg( [u_0]_N+C(t-t_0)[v]_N[u_0]_1\bigg) e^{C(t-t_0)[v]_1} \nonumber  \\
 &+ \int_{t_0}^te^{(t-s)[v]_1}\bigg( [f(s)]_N+C(t-s)[v]_N[f(s)]_1 \bigg) \,ds \, .
\end{align}
\end{proposition}
\begin{proof}
We may assume that $u$ and $f$ are two scalar functions, indeed we can work on each component of equation \eqref{nonlocale}. Note also that Theorem \ref{MP} is invariant under the time shifting $t \mapsto t+t_0 $ .

Defining \[
w:= u- \int_{t_0}^t \|f(s)\|_0 \,ds\, ,\] we have
\[
\left\{\begin{array}{l}
w_t+Lw = f-\|f(t)\|_0 \leq 0 \\ \\
w(\cdot,t_0)=u_0\, .
\end{array}\right.
\]
Thus, by Theorem \ref{MP}, we have 
\begin{equation}\label{ris1}
u(x,t) \leq \| u_0\|_0 + \int_{t_0}^t \|f(s)\|_0\, ds.
\end{equation}
Applying the same argument to the function $\tilde{w}:= u+ \int_{t_0}^t \|f(s)\|_0 \,ds\, $, we get the bound from below, showing \eqref{stabilityestimate1}.

Next, differentiate \eqref{nonlocale} in the $x$ variable to obtain
\[
(Du)_t + L \,Du = Df - DvDu\, .
\]
Applying \eqref{stabilityestimate1} to $Du$ yields
\[
[u(t)]_1 \leq [u_0]_1 +\int_{t_0}^t \bigr([f(s)]_1 + [v]_1[u(s)]_1 \bigr) \, ds \, ,
\]
and by Gronwall's inequality we get \eqref{stabilityestimate2}.
Now, differentiating \eqref{nonlocale} N times yields
\begin{equation}\label{derivata}
(D^N u)_t+L \, D^N u= D^N f + \sum_{k=0}^{N-1} c_{k,N}\, D^{k+1} u\, D^{N-k}v\, .
\end{equation}
Using again \eqref{stabilityestimate1} we can estiamte
\[
[u(t)]_N \leq [u_0]_N +\int_{t_0}^t \bigr([f(s)]_N + C ([v]_N[u(s)]_1+[v]_1[u(s)]_N) \bigr) \, ds \, ,
\]
and plugging the estimate \eqref{stabilityestimate2}, we get
\begin{align*}
[u(t)]_N &\leq [u_0]_N +C(t-t_0)[v]_N[u_0]_1 e^{(t-t_0)[v]_1}+\int_{t_0}^t \Big([f(s)]_N \\
&+ C [v]_N\int_{t_0}^s e^{(s-r)[v]_1}[f(r)]_1\,dr+C[v]_1[u(s)]_N \Big) \, ds \, ,
\end{align*}
and Gronwall's inequality finally leads to \eqref{stabilityestimate3}.
\end{proof}
Using Proposition \ref{C0stability} we also get the following
 
\begin{proposition}\label{p:transport_derivatives1}
Assume $0\leq(t-t_0)[v]_1\leq 1$. Then,  any solution $u$ of \eqref{nonlocale} satisfies 
\begin{equation}\label{e:trans_est_alpha}
\norm{u(t)}_{\alpha} \leq e^{\alpha} \left(\norm{u_0}_{\alpha} + \int_{ t_0}^t  \norm{f (\cdot, \tau)}_{\alpha}\, d\tau\right)\,,
\end{equation}
for all $0\leq \alpha\leq 1$, and, more generally, for any $N\geq 1$ and $0\leq \alpha< 1$
\begin{align}
[u (t)]_{N+\alpha} & \lesssim [ u_0]_{N+\alpha} + (t-t_0)[v]_{N+\alpha}[u_0]_1  +\int_{t_0}^t \Bigl([f (\tau)]_{N+\alpha} + (t-\tau) [v ]_{N+\alpha} [f (\tau)]_{1}\Bigr)\,d\tau\, ,
\label{e:trans_est_1}
\end{align}
where the implicit constant depends only on $N$ and $\alpha$.
\end{proposition}
\begin{proof}
For any $\alpha\in [0,1]$, let 
$$
w(x,t;h):=\frac{\delta_h u(x,t)}{|h|^\alpha}=\frac{u(x+h,t)-u(x,t)}{|h|^\alpha}\,.
$$
We have that this new function $w$ satisfies (see equation (4.13) in \cite{CTV2015})
$$
\Big( \partial_t +\nu(-\Delta)^\gamma+v\cdot \nabla_x+\delta_h v\cdot \nabla_h\Big)w=\alpha \frac{\delta_h v}{|h|}\cdot\frac{h}{|h|} w+\frac{\delta_h f}{|h|^\alpha}\,,
$$
Thus by the maximum principle\footnote{Here the maximum principle is applied in both the variables $x, h$.} \eqref{stabilityestimate1} and since $\sup_{h,x}|w(x,t;h)|=[u(t)]_\alpha$ we get
$$
[u(t)]_\alpha\leq [u_0]_\alpha+\int_{t_0}^t \big( \alpha [v(s)]_1[u(s)]_\alpha+[f(s)]_\alpha\big)\,ds\,,
$$
from which, by Gronwall's inequality, \eqref{e:trans_est_alpha} follows.

To get the higher order bounds \eqref{e:trans_est_1} just differentiate the equation $N $ times as in \eqref{derivata} and apply the previous argument with
$$
w(x,t;h):=\frac{\delta_h D^Nu(x,t)}{|h|^\alpha}=\frac{D^Nu(x+h,t)-D^Nu(x,t)}{|h|^\alpha}\,,
$$
then \eqref{e:trans_est_1} is again a consequence of \eqref{stabilityestimate1} and Gronwall's inequality.
\end{proof}

\subsection{Local existence of smooth solutions}
We want to consider exact (smooth) solutions to the fractional Navier-Stokes equations
\begin{equation}\label{NSgamma_ap}
\left\{\begin{array}{l}
\partial_t v+ v\cdot\nabla v +\nabla p +\nu \gammalaplace v=0\\  \\
\div v = 0 \\ \\
v(\cdot,0)= u_0,
\end{array}\right.
\end{equation}
in the periodic setting $\T^3\times[0,T]$, where $\gamma \in (0,1)$ and $\nu>0$. We define the space
\[ 
V^m:=\{v\in H^m(\T^3)\, : \, div\, v=0 \}\,.
\]

We start with the following 
\begin{theorem}\label{localexist}
For any  $m\geq 3$ there exists a constant $c_m=c(m)$ such that the following holds. Given any initial condition $u_0 \in V^m$ and $T_m:=c_m \|u_0\|_V^{-1}$ there exists a unique solution $v \in C([0,T_m], V^{m}) \cap C^1([0,T_m], V^{m-2})$. Moreover we have the estimate
\begin{equation}\label{estHm}
\|v(t)\|_{V^m} \leq \|u_0\|_{V^m} e^{c_m \int_0^t \|\nabla v(s)\|_0 \, ds} \qquad \forall t \in [0,T_m]\,.
\end{equation}
\end{theorem}
For a proof of Theorem \ref{localexist} we refer to \cite{MaBe2002} (Theorem 3.4 in Chapter 3). Notice that that theorem is stated for the classical Navier-Stokes equations. The proof uses the so called "energy method" and it can be easily adapted to any power $\gamma$ of the Laplacian in the equations \eqref{NSgamma_ap}.

We now want to prove that there exists a maximal time of existence (independent on $m$) of such solution. In particular, if the initial datum is smooth, we get the local existence of a smooth solution of Eq. \eqref{NSgamma_ap}. We also prove some stability estimates of such solution in H\"older spaces, since they will play a crucial role in the iterative construction.

\begin{proposition}
\label{p:local:Euler}
For any $\nu>0$ and any $0<\alpha<1$ there exists a constant $c=c(\alpha)>0$ with the following property. Given any initial data $u_0\in C^{\infty}$, and $T\leq c\norm{u_0}_{1+\alpha}^{-1}$, there exists a unique solution $v:\mathbb R^3 \times [0,T]\rightarrow \mathbb R^3$  of \eqref{NSgamma_ap}.
Moreover, $v$ obeys the bounds
\begin{align}
\norm{v}_{N+\alpha} \lesssim &\norm{u_0}_{N+\alpha}~. \label{e:euler_eq_bd_k}
\end{align}
for all $N\geq 1$, where the implicit constant depends on $N$ and $\alpha>0$.
\end{proposition}

\begin{proof}[Proof of Proposition~\ref{p:local:Euler}]
We first show that all solutions given by Theorem \ref{localexist} exist in the interval $[0,T]$, for any $T\lesssim\norm{u_0}_{1+\alpha}^{-1}$.
Fix any $\alpha\in (0,1)$ and let $T^*$ be the maximal time such that
\[
T^* \sup_{0\leq t\leq T^*}[v(t)]_1\leq 1\,.
\]
Suppose $T^*<c\|u_0\|_{1+\alpha}^{-1}$, for some constant $c=c(\alpha)$ to be fixed later (we will see that this contraddicts the assumption on the maximality of $T^*$, in particular $T^*\geq c\|u_0\|_{1+\alpha}^{-1}$).
Using Schauder estimate on $-\Delta p=tr(\nabla v  \nabla v)$ we have
\[
\|p(t)\|_{2+\alpha} \lesssim \|v(t)\|_{1+\alpha}^2\, ,
\]

thus, differentiating the equation in the $x$ variable we get
\[
\|(\partial_t +v\cdot \nabla +\nu \gammalaplace) Dv \|_\alpha \lesssim \|v(t)\|_{1+\alpha}^2\,.
\]
By Proposition \ref{p:transport_derivatives1}, for any $0\leq t\leq T^*$, we have
\[
\|v(t)\|_{1+\alpha}\lesssim \|u_0\|_{1+\alpha}+\int_0^t \|v(s)\|_{1+\alpha}^2\, ds\,.
\]
Finally, using Gronwall's inequality we get the estimate
\[
\|v(t)\|_{1+\alpha}\lesssim \|u_0\|_{1+\alpha}<\frac{1}{T^*}\, \quad \forall t \in[0,T^*]\,,
\]
where in the last inequality we have choosen the constant $c=c(\alpha)$ to get it "strict". Obviously, this contraddicts the hypothesis on the maximality of $T^*$, and also gives the a priori estimate \eqref{e:euler_eq_bd_k} for $N=1$, which together with \eqref{estHm}, gives the existence of a smooth solution in the interval $[0,T]$, for any $ T\leq c\norm{u_0}_{1+\alpha}^{-1}$.

We are left with the higher-order bounds \eqref{e:euler_eq_bd_k} for $N\geq2$. For any multi-index $\theta$ with $|\theta|=N$ we have
\[
\partial_t\partial^{\theta}v+v\cdot\nabla\partial^\theta v+\nu \gammalaplace \partial^{\theta}v +[\partial^\theta,v\cdot\nabla]v+\nabla\partial^\theta p=0.
\]
Using again Schauder estimates for the pressure we obtain
\[
\|\nabla\partial^{\theta}p\|_{\alpha}\lesssim \|{\rm tr}\, (\nabla v \nabla v)\|_{N-1+\alpha}\lesssim \|v\|_{1+\alpha}\|v\|_{N+\alpha}.
\]
Therefore 
\[
\|(\partial_t+v\cdot\nabla+\nu\gammalaplace)\partial^\theta v\|_{\alpha}\lesssim  \|v\|_{1+\alpha}\|v\|_{N+\alpha},
\]
and \eqref{e:euler_eq_bd_k} follows by applying  \eqref{e:trans_est_alpha}  and Gr\"onwall's inequality.
\end{proof}

\section{The main inductive Proposition and proofs of Theorems \ref{vanishvisc} and \ref{assegnocinetica} }
As already outlined, the main construction is taken from \cite{BDLSV2017}, thus we are not going to prove all technical details about the mechanism of the convex integration scheme. However all the proofs of the propositions involving the  structure of the Navier-Stokes equations (different from the Euler ones), are completely self contained. 

\subsection{Inductive proposition}
First of all, we impose for the moment that
\begin{equation}\label{e:energy_time_D}
\sup_{t\in [0,1]} \abs{ e'(t)}\leq 1\, 
\end{equation}
(we will see later that this can be done provided that we impose some conditions on the parameters appearing in the iteration). 

Let then $q\geq 0$ be a natural number. At a given step $q$ we assume to have a triple $(v_q,p_q,\mathring{R}_q)$ to the fractional Navier-Stokes Reynolds system, namely such that
\begin{equation}\label{NSR}
\left\{\begin{array}{l}
\partial_t v_q + \div (v_q\otimes v_q) + \nabla p_q +\nu \gammalaplace v_q =\div\mathring{R}_q\\ \\
\div v_q = 0\, ,
\end{array}\right.
\end{equation} 
to which we add the constraints 
\begin{align}
&\tr \mathring{R}_q=0 \,,\label{e:trace_free} \\
\int_{\T^3} p_q& (x,t)\, dx = 0\,.\label{e:press_const}
\end{align}
 In \eqref{NSR} the viscosity $\nu$ is just some small constant (in particular $\nu <1$) depending on some parameters of the inductive construction. In what follows we will see that this coefficient comes from a "technical rescaling" on the equations \eqref{NSgamma}.

The size of the approximate solution $v_q$ and the error $\mathring{R}_q$ will be measured by a frequency $\lambda_q$ and an amplitude $\delta_q$, which are given by
\begin{align}
\lambda_q&= 2\pi \ceil{a^{(b^q)}}\label{e:freq_def}\\
\delta_q&=\lambda_q^{-2\beta} \label{e:size_def}
\end{align}
where $\ceil{x}$ denotes the smallest integer $n\geq x$, $a>1$ is a  large parameter, $b>1$ is close to $1$ and $0<\beta<\sfrac13$ is the exponent of Theorem \ref{assegnocinetica}. The parameters $a$ and $b$ are then related to $\beta$. 

We proceed by induction, assuming the estimates
\begin{align}
\norm{\mathring R_q}_{0}&\leq  \delta_{q+1}\lambda_q^{-3\alpha}\label{e:R_q_inductive_est}\\
\norm{v_q}_1&\leq M \delta_q^{\sfrac12}\lambda_q\label{e:v_q_inductive_est}\\
\norm{v_q}_0 & \leq 1- \delta_q^{\sfrac12}\label{e:v_q_0}\\
\delta_{q+1}\lambda_q^{-\alpha} &\leq e(t)-\int_{\T^3}\abs{  v_q}^2\,dx\leq \delta_{q+1}\label{e:energy_inductive_assumption}
\end{align}
where $0 < \alpha  < 1$ is a small parameter to be chosen suitably (which will depend upon $\beta$), and $M$ is a universal constant.

\begin{proposition}\label{p:main} There exists a universal constant $M$ with the following property. Let $0<\beta<\sfrac13$, $0<\gamma<\sfrac13$ and
\begin{equation}\label{e:b_beta_rel}
1<b<min\bigg\{\frac{1-\beta}{2\beta},\frac{1-\beta}{2\gamma}, \frac{4}{3} \bigg\}\,.
\end{equation}
Then there exists an $\alpha_0$ depending only on $\beta$ and $b$, such that for any $0<\alpha<\alpha_0$ there exists an $a_0$ depending on $\beta$, $b$, $\alpha$ and $M$, such that for any $a\geq a_0 $ the following holds: given a strictly positive function $e: [0, T]\to \R^+ $ satisfying \eqref{e:energy_time_D}, and a  triple $(v_q, p_q,\mathring R_q)$ solving \eqref{NSR}-\eqref{e:press_const} and satisfying the estimates \eqref{e:R_q_inductive_est}--\eqref{e:energy_inductive_assumption}, then there exists a solution  $(v_{q+1}, p_{q+1}, \mathring R_{q+1})$ to \eqref{NSR}-\eqref{e:press_const} satisfying \eqref{e:R_q_inductive_est}--\eqref{e:energy_inductive_assumption} with $q$ replaced by $q+1$. Moreover, we have 
\begin{equation}
\norm{v_{q+1}-v_q}_0+\frac{1}{\lambda_{q+1}}\norm{v_{q+1}-v_q}_1 
\leq M\delta_{q+1}^{\sfrac12}\label{e:v_diff_prop_est}.
\end{equation}
Furthermore, $v_{q+1}(\cdot,0)$ depends only on $e(0)$ and $v_q(\cdot,0)$.
\end{proposition}

The proof of Proposition~\ref{p:main} is summarized in the Sections \ref{s:mollification}, \ref{s:gluing_outline} and \ref{s:perturbation_outline}, but its details will occupy most of the paper. We show next that this proposition immediately implies Theorem~\ref{assegnocinetica}.

\subsection{Proof of Theorem \ref{assegnocinetica}} First of all, we fix any H\"older exponent $\beta< \sfrac{1}{3}$ and also the
parameters $b$ and $\alpha$, the first satisfying \eqref{e:b_beta_rel} and the second smaller than the threshold given in Proposition \ref{p:main}. Next we show that, without loss of generality, we may further assume the energy profile satisfies 
\begin{equation}\label{e:normalized_energy}
\inf_t e(t) \geq \delta_{1}\lambda_0^{-\alpha},\qquad\sup_t e(t)\leq \delta_{1}, \quad\mbox{and}\quad \sup_t e'(t)\leq 1\, ,
\end{equation}
provided the parameter $a$ is chosen sufficiently large. 

To see this, we first make the following transformations
\begin{equation}\label{rescaling}
\tilde v(x,t):= \mu \,v(x,\mu t) \quad  \quad \tilde p(x,t) := \mu^2 p(x,\mu t)\,.
\end{equation}
Thus if we choose 
\[
\mu=\delta_{1}^{\sfrac{1}{2}},
\] 
the stated problem  reduces to finding a solution $(\tilde v, \tilde p)$ of 
\begin{equation}\label{NSviscoso}
\left\{\begin{array}{l}
\partial_t \tilde v+ \tilde v\cdot\nabla \tilde v +\nabla \tilde p +\mu \gammalaplace \tilde v=0\\ \\
\div \tilde v = 0
\end{array}\right.
\end{equation}
with the energy profile given by
\[
\tilde e(t)=\mu^2 e(\mu t)\,,
\]
for which we have (using our assumptions on the function $e(t)$)
\[
\inf_t \tilde e(t) \geq \delta_{1}\inf_t e(t) \geq \frac{\delta_1}{2}, \qquad
\sup_t \tilde e(t)\leq  \delta_{1}, \qquad \mbox{and}\qquad
\sup_t\tilde e'(t)\leq \delta_{1}^{\sfrac32} K\, .
\]
If $a$ is chosen sufficiently large, in particualar $a\geq  a_0 K^{\sfrac{1}{3 \beta}}  $,  then we can ensure
\[\sup_t\tilde e'(t)\leq \delta_{1}^{\sfrac32}  K\leq 1, \qquad\mbox{and}\qquad \frac{1}{2}\geq \lambda_0^{-\alpha}\,.\]

Now we apply Proposition \ref{p:main} iteratively with $(v_0,R_0, p_0)=(0,0,0)$. Indeed the pair $(v_0, R_0)$ trivially satisfies \eqref{e:R_q_inductive_est}--\eqref{e:v_q_0}, whereas the estimate \eqref{e:energy_inductive_assumption} and \eqref{e:energy_time_D} follows as a consequence of \eqref{e:normalized_energy}. 
Notice that by \eqref{e:v_diff_prop_est} $v_q$ converges uniformly to some continuous $\tilde v$. Moreover, we recall that the pressure
is determined by
\begin{equation}\label{e:pressure_eq}
\Delta p_q = \div \div (-v_q\otimes v_q + \mathring{R}_q)
\end{equation}
and \eqref{e:press_const} and thus $p_q$ is also converging to some pressure $\tilde p$ (for the moment only in $L^r$ for every $r<\infty$). Since $\mathring{R}_q\to 0$ uniformly, the pair $(\tilde v,\tilde p)$ solves equations \eqref{NSviscoso}. 

Observe that using \eqref{e:v_diff_prop_est} we also infer\footnote{Throughout the manuscript we use the the notation $x \lesssim y$ to denote  $x \leq C y$, for a sufficiently large constant $C>0$, which is independent of $a,b$, and $q$, but may change from line to line.} 
\begin{align*}
\sum_{q=0}^{\infty} \norm{v_{q+1}-v_q}_{\beta'} &
\lesssim \; \sum_{q=0}^{\infty} \norm{v_{q+1}-v_q}_{0}^{1-\beta'}\norm{v_{q+1}-v_q}_{1}^{\beta'}
\lesssim \; \sum_{q=0}^{\infty} \delta_{q+1}^{\frac{1-\beta'}2}\left(\delta_{q+1}^{\sfrac12}\lambda_q\right)^{\beta'} \notag
\lesssim  \sum_{q=0}^{\infty} \lambda_q^{\beta'-\beta}
\end{align*}
and hence that $v_q$ is uniformly bounded in $C^0_tC^{\beta'}_x$ for all $\beta'<\beta$. Using the last inequality and the definitions of the parameters $\lambda_q$ we also have that if $a$ is chosen sufficiently large, then
\[
 \|\tilde v\|_{\beta'} \leq 1\, , \qquad \forall \beta'<\beta.
\]
Since $\delta_{q+1}\rightarrow 0$ as $q\rightarrow \infty$, from \eqref{e:energy_inductive_assumption} we have
\begin{equation*}
\int_{\T^3}\abs{  \tilde v}^2\,dx= \tilde e(t)\, ,
\end{equation*}
If now we use the transformation
\[
 v(x,t):= \mu^{-1} \,\tilde v(x,\mu^{-1} t) \quad \mbox{and} \quad  p(x,t) := \mu^{-2} \tilde p(x,\mu^{-1} t)\, ,
\]
then it is clear that the pair $(v,p)$ solves \eqref{NSgamma_senzadato} and it satisfies \eqref{condizione1} and \eqref{condizione2}.
 To recover the time regularity we fix a smooth  standard mollifier $\psi$  in space, let $q \in \N$,  and consider $\tilde v_q := v*\psi_{2^{-q}}$, where $\psi_{\ell}(x) = \ell^{-3} \psi(x \ell^{-1})$. From standard mollification estimates we have
\begin{align}
\norm{\tilde  v_q-v}_0\lesssim \norm{v}_{\beta'} 2^{-q \beta'},
\label{e:trivial:1}
\end{align}
 and thus $\tilde v_q - v \to 0$ uniformly as $q \to \infty$. Moreover, $\tilde v_q$ obeys the following equation
 \[
 \partial_t \tilde v_{q}+\div\left(v  \otimes v \right)*\psi_{2^{-q}}  +\nabla p * \psi_{2^{-q}}+\gammalaplace \tilde v_q=0.
 \]
Next, since
\[
-\Delta p *\psi_{2^{-q}} = \div \div (  v\otimes v )* \psi_{2^{-q}} \, ,
\]
using Schauder's estimates, for any fixed $\varepsilon>0$ we get
\[
\|\nabla p * \psi_{2^{-q}}\|_0 \leq \|\nabla p * \psi_{2^{-q}}\|_\varepsilon \lesssim \|v\otimes v\|_{\beta'} 2^{q(1 + \eps - \beta')}
\lesssim \|v\|^2_{\beta'} 2^{q(1 + \eps - \beta')}\, ,
\]
(where the constant in the estimate depends on $\varepsilon$ but not on $q$). Similarly, 
\[
\norm{ \left(v  \otimes v \right)*\psi_{2^{-q}}}_1 \lesssim\norm{v\otimes v}_{\beta'}  2^{q(1-\beta')}  \lesssim \norm{v}_{\beta'}^2  2^{q(1-\beta')} \,
\]
\[
\| \gammalaplace \tilde v_q\|_0 \leq \| \tilde v_q\|_1 \lesssim \norm{v}_{\beta'}  2^{q(1-\beta')} \, .
\]
Thus the above estimates yield
\begin{align}
\norm{\partial_t \tilde  v_{q}}_0 &\lesssim  \|v\|^2_{\beta'} 2^{q(1 + \eps - \beta')}\,.
\label{e:trivial:2}
\end{align}
Next, for $\beta'' < \beta'$ we conclude from \eqref{e:trivial:1} and \eqref{e:trivial:2} that
\begin{align*}
\norm{\tilde v_q - \tilde v_{q+1}}_{C^0_x C^{\beta''}_t}  
&\lesssim \left( \norm{\tilde v_q - v}_0 + \norm{\tilde v_{q+1} - v}_0 \right)^{1-\beta''}
\left(\norm{\partial_t \tilde v_q}_0  + \norm{\partial_t \tilde v_{q+1}}_0\right)^{\beta''}
\notag\\
&\lesssim \norm{v}_{\beta'}^{1 + \beta''} 2^{- q \beta' (1-\beta'')} 2^{q \beta''(1+ \eps - \beta')} = \norm{v}_{\beta'}^{1 + \beta''} 2^{- q (\beta' - (1+ \eps) \beta'')} \notag\\
&\lesssim \norm{v}_{\beta'}^{1 + \beta''} 2^{-q \eps}
\end{align*}
Here we have chosen $\eps>0$ sufficiently small (in terms of $\beta'$ and $\beta''$) so that  that $\beta' - (1+ \eps) \beta'' \geq \eps $. Thus, the series 
\[
v = \tilde v_0 + \sum_{q\geq0} (\tilde v_{q+1} - \tilde v_q)
\]
converges in $C^0_x C^{\beta''}_t$. Since we already know $v \in C^0_t C^{\beta'}_x$, we obtain that $v \in C^{\beta''}([0,1] \times \T^3)$ as desired, with $\beta'' < \beta' < \beta < 1/3$ arbitrary.

This concludes the proof of the theorem. 
\subsection{Proof of Theorem \ref{vanishvisc}}
Let $v \in C_{t,x}^{\beta'}$ be a dissipative solution of Euler, with the kinetic energy profile satisfying the assumptions $(i)-(v)$ in the proof of Theorem \ref{t:locale} (note that the proof of the existence of such solution is given in \cite{BDLSV2017}). Using the rescaling \eqref{rescaling}, with $\mu:= (2 \|v\|_0)^{-1}$, we can assume that $\|v\|_0\leq \sfrac{1}{2}$.

We fix two positive kernels (Friedrichs mollifiers) $\varphi$ and $\psi$, respectively in space and time. Let $\delta_n:=a^{-b^{n+2}}$ and $\nu_n:=\delta_n^{1+\beta'}$. Since $v$ solves Euler,  the smooth function $v_n:=(v \ast  \varphi_{\delta_n}) \ast \psi_{\delta_n}$ solves the following Navier-Stokes Reynolds equations
\[
\partial_t v_n+\div(v_n \otimes v_n)+\nabla p_n + \nu_n \gammalaplace v_n= \div \mathring{R}_n\, ,
\]
with 
\[
 \mathring{R}_n=v_n\mathring{\otimes} v_n-(v \mathring{\otimes} v)_n + \nu_n \mathcal{R}\gammalaplace v_n\, ,
\]
where $f \mathring{\otimes} g$ is the traceless part of the matrix $ f \otimes g$ and $\mathcal{R}$ is the operator defined in \eqref{e:R:def}. We also define the energy as
\begin{equation}\label{energian}
e_n(t):= \int_{\T^3}|v_n|^2 \, dx + \delta_{n+1}\lambda_n^{-\alpha}\,.
\end{equation}
Using standard mollification estimates and Proposition \ref{p:CET} we have
\begin{align*}
\|v_n\|_1 &\lesssim \delta_n^{\beta'-1}\,, \\
\| \mathring{R}_n\|_0 &\lesssim \delta_n^{2 \beta'}+\nu_n [v_n]_1 \lesssim \delta_n^{2 \beta'}\, .
\end{align*}
Thus, if we chose $\gamma <\beta<\beta'$ and the parameter $a$ large enough, we can guarantee that \eqref{e:R_q_inductive_est}-\eqref{e:energy_inductive_assumption} hold for $q=n$, provided that $b$ is sufficiently near $1$ and $\alpha$ is small.

We can now apply Proposition \ref{p:main} (inductively for $q\geq n$) in order to obtain a solution $v^{(\nu_n)}$ of \eqref{NSvisc}, and since $\gamma <\beta$ (as already done in the proof of Theorem \ref{t:locale}) we can guarantee that $v^{(\nu_n)}$ is indeed a Leray-Hopf weak solution.

Moreover by \eqref{e:v_diff_prop_est} we have
\[
\|v^{(\nu_n)}-v_n\|_{\beta''}\leq \sum_{q\geq n} \|v_{q+1}-v_q\|_{\beta''} \lesssim \sum_{q\geq n} a^{(\beta''-\beta)b^{q+1}}.
\]
Thus, provided that the parameter $a$ is chosen even larger, we can ensure that 
\[
\|v^{(\nu_n)}-v\|_{\beta''}\leq \|v^{(\nu_n)}-v_n\|_{\beta''}+\|v_n-v\|_{\beta''}\leq \frac{1}{n}\,, \qquad \forall \beta'' <\beta\, ,
\] 
and this concludes the proof of the theorem. We also remark that $e_n(t)\rightarrow\int_{\T^3} |v|^2\,dx$ as $n\rightarrow +\infty$.

\section{The convex integration scheme and proof of the iterative Proposition}

The rest of the paper is devoted to the proof of Proposition \ref{p:main}. To simplify several estiamtes we will assume that $\alpha$ is small enough so to have
\begin{align}\label{e:some_param_ineq}
\lambda_{q}^{3\alpha}\leq\left(\frac{\delta_q}{\delta_{q+1}}\right)^{\sfrac32}\leq \frac{\lambda_{q+1}}{\lambda_q}\, ,
\end{align}
in which we also need that $a$ is big enough to nullify any constant from the ratio $\lambda_q/a^{(b^q)}$, which can be easily bounded as
\begin{equation}\label{e:bloody_integers}
2\pi \leq \frac{\lambda_q}{a^{b^q}}\leq 4\pi\, .
\end{equation}

Following the construction of \cite{BDLSV2017} we subdivide the proof in three stages, in each of which we modify $v_q$: mollification, gluing and perturbation.
\subsection{Mollification step}\label{s:mollification}
The first stage is mollification: we mollify $v_q$ (in space) at length scale 
\begin{equation}\label{e:ell_def}
\ell:=\frac{\delta_{q+1}^{\sfrac 12}}{\delta_q^{\sfrac12 }\lambda_q^{1+\sfrac{3\alpha}{2}}}\, .
\end{equation}
Fix a standard mollification kernel  $\psi$, we define
\begin{align*}
v_{\ell}:=& v_q*\psi_\ell\\
\mathring{R}_{\ell}:=& \mathring R_q*\psi_\ell  -(v_q\mathring\otimes v_q)*\psi_\ell  + v_{\ell}\mathring\otimes v_{\ell} \,.
\end{align*}
These functions obey the equation
\begin{equation}\label{e:euler_reynolds_l}
\left\{\begin{array}{l}
\partial_t v_\ell + \div (v_\ell\otimes v_\ell) + \nabla p_\ell+\nu \gammalaplace v_\ell =\div\mathring{R}_\ell\\ \\
\div v_\ell = 0\, ,
\end{array}\right.
\end{equation} 
in view of \eqref{NSR}.

Observe, again choosing $\alpha$ sufficiently small and $a$ sufficiently large we can assume
\begin{equation}\label{e:compare-lambda-ell}
\lambda_q^{-3/2}\leq\ell \leq \lambda_q^{-1}\,,
\end{equation}
which will be used in order to simplify  several estimates.

From  standard mollification estimates we  obtain the following bounds\footnote{In the following, when considering higher order norms $\|\cdot\|_N$ or $\|\cdot\|_{N+1}$, the symbol $\lesssim$ will imply that the constant in the inequality might also depend on $N$.}
 (we refer to \cite{BDLSV2017} for a detailed proof). 
\begin{proposition}\label{p:est_mollification}
\begin{align}
\norm{v_{\ell}-v_q}_0&\lesssim \delta_{q+1}^{\sfrac12}\lambda_q^{-\alpha}\,,\label{e:v:ell:0}\\
\norm{v_{\ell}}_{N+1} &\lesssim \delta_q^{\sfrac 12}\lambda_q\ell^{-N}\qquad \forall N \geq 0\,,  \label{e:v:ell:k}\\
\norm{\mathring{R}_{\ell}}_{N+\alpha}&\lesssim  \delta_{q+1}\ell^{-N+\alpha} \qquad \forall N\geq 0\,. \label{e:R:ell}\\
\abs{\int_{\T^3}\abs{v_q}^2-\abs{v_{\ell}}^2\,dx} &\lesssim \delta_{q+1}\ell^{\alpha}\,.
\label{e:vq_vell_energy_diff}
\end{align}
\end{proposition}

\subsection{Gluing step}\label{s:gluing_outline}
In the second stage we glue together exact solutions to the fractional Navier-Stokes equations in order to produce a new $\overline v_q$, close to $v_q$, whose associated Reynolds stress error has support in pairwise disjoint temporal regions of length $\tau_q$ in time, where
\begin{equation}\label{e:tau_def}
\tau_q= \frac{\ell^{2\alpha}}{\delta_q^{\sfrac 12}\lambda_q}.
\end{equation} 

Note that we have the CFL-like condition
\begin{equation}
2\tau_q \norm{v_{\ell}}_{1+\alpha} \stackrel{\eqref{e:v:ell:k}}{\lesssim} \tau_q\delta_q^{\sfrac{1}{2}} \lambda_q \ell^{-\alpha}  \lesssim \ell^{\alpha} \ll 1 \label{e:CFL}
\end{equation}
as long as $a$ is sufficiently large.

More precisely, we aim to construct a new triple $(\overline v_q,  \mathring{\overline R_q},\overline p_q)$ solving the Navier-Stokes Reynolds equation \eqref{NSR} such that the temporal support of $ \mathring{\overline R_q} $ is contained in pairwise disjoint intervals $I_i$ of length $\sim\tau_q$ and such that the gaps between neighbouring intervals is also of length $\sim\tau_q$.

For each $i$, let $t_i= i \tau_q$, and consider smooth solutions of the fractional Navier-Stokes equations
\begin{equation}
\left\{\begin{array}{l}
\partial_t v_i + \div (v_i \otimes v_i) + \nabla p_i+ \nu \gammalaplace v_i=0\\ \\
\div v_i = 0\\ \\
v_i(\cdot,t_i)=v_{\ell}(\cdot, t_i)
\, .
\end{array}\right.
\label{e:vi:def}
\end{equation} 
defined over their own maximal interval of existence. An immediate consequence of \eqref{e:v:ell:k}, \eqref{e:tau_def} and Proposition \ref{p:local:Euler} is the following
\begin{corollary}
\label{c:size:vi}
If $a$ is sufficiently large, for $0\leq(t-t_i)\leq 2\tau_q$,  we have
\begin{equation}\label{e:z:N+alpha}
\norm{v_i}_{N+\alpha} \lesssim \delta_q^{\sfrac12}\lambda_q\ell^{1-N-\alpha}\lesssim \tau_q^{-1}\ell^{1-N+\alpha}\,\qquad
\mbox{for any $N\geq 1$}.
\end{equation}
\end{corollary}
We will now show that for $0\leq(t-t_i)\leq 2\tau_q$, $v_i$ is close to $v_{\ell}$ and by the identity 
\begin{equation*}
v_i - v_{i+1} = (v_{i}-v_{\ell})-(v_{i+1}-v_{\ell}),
\end{equation*}
the vector field $v_i$ is also close to $v_{i+1}$.

\begin{proposition}[Stability and estimates on \texorpdfstring{$v_i-v_\ell$}{vi-vl}]
\label{p:vi:vell}
For $0\leq(t-t_i)\leq 2\tau_q$, $N\geq 0$ and $0<\nu<1$ we have
\begin{align}
\norm{v_i-v_{\ell}}_{N+\alpha} \lesssim & \tau_q\delta_{q+1}\ell^{-N-1+\alpha}\,, \label{e:z_diff_k}\\
\norm{\nabla(p_{\ell} - p_i)}_{N+\alpha} &\lesssim \delta_{q+1}\ell^{-N-1+\alpha}\,, \label{e:pressure_1}\\
\norm{L_{t,\ell,\gamma}(v_i-v_{\ell})}_{N+\alpha} &\lesssim  \delta_{q+1}\ell^{-N-1+\alpha}\,, \label{e:z_D_t_lapl}\\
\norm{D_{t,\ell}(v_i-v_{\ell})}_{N+\alpha} &\lesssim  \delta_{q+1}\ell^{-N-1+\alpha}\,, \label{e:z_D_t}
\end{align}
where we write
\begin{align}
D_{t,\ell}=\partial_t+v_{\ell}\cdot\nabla \qquad L_{t,\ell,\gamma}=D_{t,\ell}+\nu \gammalaplace\,.
\label{e:Dtell:def}
\end{align}

\end{proposition}

\begin{proof}
Let us first consider \eqref{e:z_diff_k} with $N=0$.
From \eqref{e:euler_reynolds_l} and \eqref{e:vi:def} we have
\begin{equation}\label{e:z_diff_evo}
L_{t,\ell,\gamma}(v_\ell-v_i)
= (v_i - v_{\ell}) \cdot \nabla v_i - \nabla (p_{\ell} - p_i)+\div \mathring{R}_{\ell}.
\end{equation}
In particular, using
\begin{equation}\label{e:eqnpi}
\Delta (p_{\ell} - p_i) = \div\bigl(\nabla v_\ell(v_i-v_\ell)\bigr)+\div\bigl(\nabla v_i{(v_i-v_\ell)}\bigr)+\div\div\mathring{R}_{\ell},
\end{equation}
estimates \eqref{e:R:ell} and \eqref{e:z:N+alpha}, and Proposition~\ref{p:CZO_C_alpha} (recall that $\partial_{i} \partial_j (-\Delta)^{-1}$ is given by $\sfrac 13 \delta_{ij} ~+ $ a Calder{\'o}n-Zygmund operator), we conclude
\begin{equation*}
\norm{\nabla (p_{\ell} - p_i) (\cdot, t)}_{\alpha} \leq  \delta_q^{\sfrac12}\lambda_q\ell^{-\alpha}\norm{v_{i} - v_\ell}_{\alpha}+\delta_{q+1}\ell^{-1+\alpha}\,.
\end{equation*}
Thus, using \eqref{e:R:ell} and the definition of $\tau_q$, we have
\begin{equation}\label{e:z_D_t_step}
\norm{L_{t,\ell,\gamma} (v_{\ell} - v_i) }_{\alpha}\lesssim \delta_{q+1}\ell^{-1+\alpha}+\tau_q^{-1} \norm{v_{\ell} - v_i}_{\alpha}
\end{equation}
By applying  \eqref{e:trans_est_alpha} we obtain
\begin{align*}
\norm{ (v_{\ell} - v_i) (\cdot, t) }_{\alpha} \lesssim  \abs{t-t_i}\delta_{q+1}\ell^{-1+\alpha}+ \int_{t_i}^t\tau_q^{-1} \norm{(v_{\ell} - v_i)(\cdot,s)}_{\alpha} ~ds.
\end{align*}
Applying Gr\"onwall's inequality and using the assumption $0\leq(t-t_i)\leq 2\tau_q$  we obtain 
\begin{equation}\label{e:zdiff:N=0}
\norm{v_i-v_{\ell}}_{\alpha} \lesssim \tau_q\delta_{q+1}\ell^{-1+\alpha}\,,
\end{equation}
i.e. \eqref{e:z_diff_k} for the case $N=0$. Then as a consequence of \eqref{e:z_D_t_step} we obtain \eqref{e:z_D_t_lapl} for $N=0$.

Next, consider the case $N\geq 1$ and let $\theta$ be a multiindex with $|\theta|=N$. 
Commuting the derivative $\partial^\theta$ with the material derivative $\partial_t + v_{\ell} \cdot \nabla$ we have
\begin{align*}
\|L_{t,\ell,\gamma}\partial^\theta (v_\ell-v_i)\|_{\alpha}&\lesssim \|\partial^\theta L_{t,\ell,\gamma} (v_\ell-v_i)\|_{\alpha}+\|[v_\ell\cdot\nabla,\partial^\theta](v_\ell-v_i)\|_\alpha\\
&\lesssim  \|\partial^\theta L_{t,\ell,\gamma}(v_\ell-v_i)\|_{\alpha}+\|v_\ell\|_{N+\alpha}\|v_\ell-v_i\|_{1+\alpha}+\|v_\ell\|_{1+\alpha}\|v_\ell-v_i\|_{N+\alpha}\\
&\lesssim \|\partial^\theta L_{t,\ell,\gamma}(v_\ell-v_i)\|_{\alpha}+\|v_\ell\|_{N+1+\alpha}\|v_\ell-v_i\|_{\alpha}+\|v_\ell\|_{1+\alpha}\|v_\ell-v_i\|_{N+\alpha}\,,
\end{align*}
On the other hand
differentiating \eqref{e:z_diff_evo} leads to
\begin{align}
\|\partial^\theta L_{t,\ell,\gamma}(v_\ell-v_i)\|_{\alpha}&\lesssim \|v_\ell-v_i\|_{N+\alpha}\|v_i\|_{1+\alpha}+\|v_\ell-v_i\|_{\alpha}\|v_i\|_{N+1+\alpha}+\|p_\ell-p_i\|_{N+1+\alpha}+\|\mathring{R}_\ell\|_{N+1+\alpha}\notag\\
&\lesssim \tau_q^{-1} \|v_\ell-v_i\|_{N+\alpha}+\delta_{q+1}\ell^{-N-1+\alpha}+\|\nabla(p_\ell-p_i)\|_{N+\alpha}\label{e:dbetaDt}\,,
\end{align}
where we have used \eqref{e:zdiff:N=0}.
Furthermore, from \eqref{e:eqnpi} we also obtain, using Corollary \ref{c:size:vi} and \eqref{e:zdiff:N=0}
\begin{align}
\|\nabla(p_\ell-p_i)\|_{N+\alpha}&\lesssim (\|v_\ell\|_{N+1+\alpha}+\|v_i\|_{N+1+\alpha})\|v_\ell-v_i\|_{\alpha} \notag\\
&\qquad +(\|v_\ell\|_{1+\alpha}+\|v_i\|_{1+\alpha})\|v_\ell-v_i\|_{N+\alpha}+\|\mathring{R}_\ell\|_{N+1+\alpha}\notag\\
&\lesssim \delta_{q+1}\ell^{-N-1+\alpha}+\tau_q^{-1}\|v_\ell-v_i\|_{N+\alpha}\label{e:pressureNN}\,.
\end{align}
Summarizing, for any multiindex $\theta$ with $|\theta|=N$ we obtain
$$
\|L_{t,\ell,\gamma}\partial^\theta (v_\ell-v_i)\|_{\alpha}\lesssim \delta_{q+1}\ell^{-N-1+\alpha}+\tau_q^{-1}\|v_\ell-v_i\|_{N+\alpha}.
$$
Therefore, invoking once more  \eqref{e:trans_est_alpha} we deduce 
\begin{equation*}
  \|(v_\ell-v_i)(\cdot,t)\|_{N+\alpha}\lesssim \tau_q\delta_{q+1}\ell^{-N-1+\alpha}+\int_{t_i}^t\tau_q^{-1}\|(v_\ell-v_i)(\cdot,s)\|_{N+\alpha}\,ds	,
\end{equation*}
and hence, using Gr\"onwall's inequality and the assumption $0\leq(t-t_i)\leq 2\tau_q$  we obtain \eqref{e:z_diff_k}. From  \eqref{e:pressureNN} and \eqref{e:dbetaDt} we then also conclude \eqref{e:pressure_1} and \eqref{e:z_D_t_lapl}. We are only left with \eqref{e:z_D_t}. 

By Theorem \ref{lapla.holder} and estimate \eqref{e:z_diff_k} we have 
\[
\nu \| \gammalaplace (v_\ell - v_i)\|_{N+\alpha} \lesssim \| v_\ell - v_i\|_{N+2\gamma +2\alpha} \lesssim  \tau_q\delta_{q+1}\ell^{-N-1-2\gamma -\alpha}\, .
\]
If $a$ is chosen sufficiently large we can ensure $\ell ^{-1} \leq \lambda_{q+1}$ and, using \eqref{e:b_beta_rel}, we get 
\[
\tau_q \ell^{-2\gamma -2\alpha}\leq \frac{\lambda_{q+1}^{2\gamma}}{\delta_q^{\sfrac{1}{2}} \lambda_q} \leq 1\, 
\]
from which we deduce
\begin{equation}\label{ultima}
\nu \| \gammalaplace (v_\ell - v_i)\|_{N+\alpha} \lesssim \delta_{q+1}\ell^{-N-1 +\alpha}\, .
\end{equation}
Finally, combining \eqref{e:z_D_t_lapl}, \eqref{ultima} and triangular inequality, we get \eqref{e:z_D_t}
\end{proof}
%%%%%%%%%%%%%%%%

Define the vector potentials to the solutions $v_i$ as
\begin{equation}\label{e:ziBidef}
z_i=\RRc v_i:=(-\Delta)^{-1}\curl v_i,
\end{equation}
where $\RRc$ is the Biot-Savart operator, so that
\begin{equation}\label{e:Biot_Savart}
\div z_i=0\qquad\textrm{ and }\qquad\curl z_i=v_i  - \int_{\T^3} v_i\, .
\end{equation}
Our aim is to obtain estimates for the differences $z_i-z_{i+1}$. 
\begin{proposition}[Estimates on vector potentials]\label{p:S_est}
For $0\leq(t-t_i) \leq 2\tau_q$, we have that
\begin{align}
\norm{z_i-z_{i+1}}_{N+\alpha} &\lesssim  \tau_q\delta_{q+1}\ell^{-N+\alpha}\,,   \label{e:z_diff} \\
\norm{D_{t,\ell} (z_i-z_{i+1})}_{N+\alpha} &\lesssim  \delta_{q+1}\ell^{-N+\alpha}\,. \label{e:z_diff_Dt}
\end{align}
\end{proposition}

\begin{proof}
Set $\tilde z_i := \RRc (v_i-v_{\ell})$ and observe that
$z_{i}-z_{i+1}=\tilde z_i-\tilde z_{i+1}$. Hence, it suffices to estimate $\tilde z_i$ in place of $z_i-z_{i+1}$.

The estimate on $\norm{\nabla \tilde z_i}_{N-1+\alpha}$ for $N\geq 1$ follows directly from \eqref{e:z_diff_k} and the fact that $\nabla \RRc$ is a bounded operator on H\"older spaces:
\begin{align}\label{e:N>=1}
\norm{\nabla \tilde z_i}_{N-1+\alpha}\leq& \norm{\nabla \RRc (v_i-v_{\ell})}_{N-1+\alpha} \norm{v_i-v_{\ell}}_{N+\alpha}
\lesssim  \tau_q\delta_{q+1}\ell^{-N+\alpha} ~.
\end{align}
Next, observe that 
\begin{equation}\label{e:eqn-vivell}
\partial_t(v_i-v_\ell)+v_\ell\cdot\nabla(v_i-v_\ell)+(v_i-v_\ell)\cdot\nabla v_i+\nabla (p_i-p_\ell)+\nu \gammalaplace(v_i-v_\ell)+\div\mathring{R}_\ell=0.	
\end{equation}
Since $v_i-v_\ell=\curl\tilde z_i$ with $\div\tilde z_i=0$, 
we have\footnote{Here we use the notation $[(z\times\nabla)v]^{ij}=\epsilon_{ikl} z^k\partial_lv^j$ for vector fields $z,v$.}
\begin{align*}
 v_\ell\cdot\nabla(v_i-v_{\ell}) &=\curl\bigl((v_\ell\cdot\nabla)\tilde z_i\bigr)+\div\bigl((\tilde z_i\times \nabla)v_\ell\bigr)\\
	((v_i-v_\ell)\cdot\nabla) v_i &=\div\bigl((\tilde z_i\times \nabla)v_i^T\bigr),
\end{align*}
so that we can write \eqref{e:eqn-vivell} as
\begin{equation}\label{e:curlz_i-eqn}
\curl(\partial_t\tilde z_i+(v_\ell\cdot\nabla)\tilde z_i+\nu \gammalaplace\tilde z_i)=-\div\bigl((\tilde z_i\times\nabla)v_\ell+(\tilde z_i\times\nabla)v_i^T\bigr)-\nabla(p_i-p_\ell)-\div\mathring{R}_\ell.	
\end{equation}
Taking the curl of \eqref{e:curlz_i-eqn} the pressure term drops out. Using in addition that $\div\tilde z_i=\div (v_i-v_\ell) =0$ and the identity 
$\curl\curl =-\Delta+\nabla\div$, we then arrive at
\begin{equation*}
	-\Delta\bigl(\partial_t\tilde z_i+(v_\ell\cdot\nabla)\tilde z_i+\nu \gammalaplace\tilde z_i\bigr)=F,
\end{equation*}
where 
$$
F=-\nabla\div\left((\tilde z_i\cdot \nabla)v_\ell\right)-\curl\div\left((\tilde z_i\times\nabla)v_\ell+(\tilde z_i\times\nabla)v_i^T\right)-\curl\div\mathring{R}_\ell.\footnote{In deriving the latter equality we have used the identity $\nabla \div ((v_\ell \cdot \nabla) \tilde{z}_i) = \nabla \div ((\tilde{z}_i \cdot \nabla) v_\ell)$, which follows easily from the fact that both $v_\ell$ and $\tilde{z}_i$ are divergence free.}
$$
Consequently, 
\begin{align}
	\|\partial_t\tilde z_i+(v_\ell\cdot\nabla)\tilde z_i+\nu \gammalaplace\tilde z_i\|_{N+\alpha} \lesssim \; &(\|v_i\|_{N+1+\alpha}+\|v_\ell\|_{N+1+\alpha})\|\tilde z_i\|_\alpha\notag\\
	&+(\|v_i\|_{1+\alpha}+\|v_\ell\|_{1+\alpha})\|\tilde z_i\|_{N+\alpha}+\|\mathring{R}_\ell\|_{N+\alpha}\notag\\
	\lesssim\; & \tau_q^{-1}\|\tilde z_i\|_{N+\alpha}+\tau_q^{-1}\ell^{-N}\|\tilde z_i\|_\alpha+\delta_{q+1}\ell^{-N+\alpha}.
	\label{e:not_yet_commuted}
\end{align}
Setting $N=0$ and using \eqref{e:trans_est_alpha} and Gr\"onwall's inequality we obtain
$$
\|\tilde z_i\|_{\alpha}\lesssim \tau_q \delta_{q+1}\ell^{\alpha}\, ,
$$
which together with \eqref{e:N>=1} gives \eqref{e:z_diff}. 
Using \eqref{e:z_diff} into \eqref{e:not_yet_commuted} we get
$$
\|\partial_t\tilde z_i+(v_\ell\cdot\nabla)\tilde z_i+\nu \gammalaplace\tilde z_i\|_{N+\alpha} \lesssim \delta_{q+1}\ell^{-N+\alpha}\, .
$$
Thus we conclude
\begin{align*}
\|\partial_t\tilde z_i+(v_\ell\cdot\nabla)\tilde z_i\|_{N+\alpha} & \lesssim \delta_{q+1}\ell^{-N+\alpha}+ \|\gammalaplace\tilde z_i\|_{N+\alpha}\lesssim \delta_{q+1}\ell^{-N+\alpha}+\| \tilde z_i\|_{N+2\gamma+2\alpha} \\
&\lesssim \delta_{q+1}\ell^{-N+\alpha}\bigr( 1+ \tau_q\ell^{-2\gamma-2\alpha} \bigr)\leq \delta_{q+1}\ell^{-N+\alpha}\,.
\end{align*}

\end{proof}

Proceding as in \cite{BDLSV2017}, we now  glue the solutions $v_i$ together in order to construct $\overline v_q$. Let
\begin{align*}
& t_i=i\tau_q,\qquad I_i=[t_{i+1}+\tfrac{1}{3}\tau_q,t_{i+1}+\tfrac{2}{3}\tau_q]\cap [0,T],\\
&J_0=[0,t_1+\tfrac{1}{3}\tau_q)  ,\qquad   J_i=(t_{i+1}-\tfrac{1}{3}\tau_q,t_{i+1}+\tfrac{1}{3}\tau_q)\cap [0,T]\quad i\geq 1 \, .
\end{align*}
Note that $\{I_i,J_i\}_i$ is a decomposition of $[0,T]$ into pairwise disjoint intervals.  Note also that this definitions of $J_i, I_i$ is slightly different from the one used in \cite{BDLSV2017}. The reason is that our stability estimates for smooth solutions of the fractional Navier-Stokes equations hold for $0\leq t-t_i\leq \tau_q$ as opposed to $|t-t_i|\leq \tau_q$ in \cite{BDLSV2017}.

We define a partition of unity $\{\chi_i\}_i$ in time with the following properties:
\begin{itemize}
\item The cut-offs form a partition of unity
\begin{equation}
\sum_i \chi_i \equiv 1
\label{e:chi:partition}
\end{equation}
\item $\supp \chi_i\cap \supp \chi_{i+2}=\emptyset$ and moreover
\begin{align}\label{e:chi:time:width}
\supp \chi_0&\subset [0,t_1+\tfrac{2}{3} \tau_q) \nonumber \\
\supp \chi_i&\subset I_{i-1} \cup J_i \cup I_i \\
\chi_i(t)&=1\quad\textrm{ for }t\in J_i
\end{align}
\item For any $i$ and $N$ we have
\begin{equation}
\norm{\partial_t^N \chi_i}_0 \lesssim \tau_q^{-N} \label{e:dt:chi}\,.
\end{equation}
\end{itemize}

%%%%%%%%%%%%%%%%%%%%%%%%%%%%%
%%%%%%%%%%%%%%%%%%%%%%%%%%%%%

We define
\begin{align*}
\overline v_q&=\sum_i \chi_i v_i\\
\overline p_q^{(1)}&=\sum_i \chi_i p_i
\end{align*}
Observe that $\div \overline v_q=0$. Furthermore, if $t\in I_i$, then $\chi_i+\chi_{i+1}=1$ and $\chi_j=0$ for $j\neq i,i+1$, therefore
on $I_i$:
\begin{align*}
\overline v_q&=\chi_i v_i+(1-\chi_i) v_{i+1}\\
\overline p_q^{(1)}&=\chi_i p_i+(1-\chi_i) p_{i+1}
\end{align*}
and
\begin{align*}
\partial_t\overline v_q+\div(\overline v_q\otimes \overline v_q)+\nabla\overline p_q^{(1)}+\nu \gammalaplace \overline v_q=\partial_t\chi_i(v_i-v_{i+1})-\chi_i(1-\chi_i)\div\left((v_i-v_{i+1})\otimes (v_i-v_{i+1})\right).
\end{align*}
On the other hand, if $t\in J_i$ then $\chi_i=1$ and $\chi_j(\tilde t)=0$ for all $j\neq i$ for all $\tilde t$ sufficiently close to $t$ (since $J_i$ is open). Then for all $t\in J_i$ we have
$$
\overline v_q=v_i,\quad \overline p_q^{(1)}=p_i, 
$$
and, from \eqref{e:vi:def}, 
$$
\partial_t\overline v_q+\div(\overline v_q\otimes \overline v_q)+\nabla\overline p_q^{(1)}+\gammalaplace \overline v_q=0.
$$

In order to define the new Reynolds tensor, we recall the operator $\mathcal R$ from \cite{DlSz2013}, which
can be thought of as an ``inverse divergence'' operator for symmetric tracefree 2-tensors. The operator is defined as
\begin{equation}
\label{e:R:def}
\begin{split}
({\mathcal R} f)^{ij} &= {\mathcal R}^{ijk} f^k \\
{\mathcal R}^{ijk} &= - \frac 12 \Delta^{-2} \partial_i \partial_j \partial_k - \frac 12 \Delta^{-1} \partial_k \delta_{ij} +  \Delta^{-1} \partial_i \delta_{jk} +  \Delta^{-1} \partial_j \delta_{ik}.
\end{split}
\end{equation}
when acting on vectors $f$ with zero mean on $\T^3$ and has the propery that $\mathcal R f$ is symmetric and $\div ( {\mathcal R}  f) = f$.

Thus we define
\begin{align*}
\mathring{\overline{R}}_q&=\partial_t\chi_i\mathcal{R}(v_i-v_{i+1})-\chi_i(1-\chi_i)(v_i-v_{i+1})\mathring{\otimes} (v_i-v_{i+1})\\
\overline{p}_q^{(2)}&=-\chi_i(1-\chi_i)\left(|v_i-v_{i+1}|^2-\int_{\T^3}|v_i-v_{i+1}|^2\,dx\right),
\end{align*}
for $t\in I_i$ and $\mathring{\overline{R}}_q=0$, $\overline{p}_q^{(2)}=0$ for $t\notin\bigcup_{i}I_i$.

Furthermore, we set
$$
\overline{p}_q=\overline{p}_q^{(1)}+\overline{p}_q^{(2)}
$$
It follows from the preceding discussion and the definition of the operator $\mathcal{R}$  that
\begin{itemize}
\item $\mathring{\overline{R}}_q$ is a smooth symmetric and traceless 2-tensor;
\item For all $(x,t)\in \T^3\times [0,T]$
\begin{equation*}
\left\{\begin{array}{l}
\partial_t\overline{v}_q+\div(\overline{v}_q\otimes\overline{v}_q)+\nabla \overline{p}_q +\nu \gammalaplace \overline v_q=\div \mathring{\overline{R}}_q,\\ \\
\div \overline{v}_q =0;
\end{array}\right.
\end{equation*}
\item $\supp\mathring{\overline{R}}_q\subset \T^3\times \bigcup_iI_i$.
\end{itemize}

Next, we estimate the various H\"older norms of $\overline{v}_q$ and $\mathring{\overline{R}}_q$. 

\begin{proposition}[Estimates on \texorpdfstring{$\overline{v}_q$}{overline vq} and $\mathring{\overline{R}}_q$]\label{p:vq:vell}
The velocity field $\overline v_q$ and the new Reynolds stress tensor $\mathring{\overline{R}}_q$ satisfy the following estimates
\begin{align}
\norm{\bar v_q - v_{\ell}}_{\alpha} &\lesssim \delta_{q+1}^{\sfrac12}\ell^{\alpha} \label{e:vq:vell} \\
\norm{\overline{v}_q-v_\ell}_{N+\alpha} &\lesssim \tau_q\delta_{q+1}\ell^{-1-N+\alpha} \label{e:vq:vell:additional} \\
\norm{\bar v_q}_{1+N} &\lesssim \delta_{q}^{\sfrac12} \lambda_q \ell^{-N}\label{e:vq:1} \\
\norm{\mathring{\overline R_q}}_{N+\alpha} &\lesssim \delta_{q+1}\ell^{-N+\alpha} \label{e:Rq:1}\\
\norm{(\partial_t + \overline v_q\cdot \nabla) \mathring{\overline R_q}}_{N+\alpha} &\lesssim \delta_{q+1}\delta_q^{\sfrac12}\lambda_q\ell^{-N-\alpha}. \label{e:Rq:Dt}
\end{align}
for all $N \geq 0$. Moreover the difference of the energies of $\overline v_q$ and $v_\ell$ satisfies
\label{p:glued:energy}
\begin{align}
\left| \int_{\T^3} |\bar v_q|^2 - |v_\ell |^2 dx \right| \lesssim \delta_{q+1}\ell^\alpha \,.\label{e:voverline_vell_energy_diff}
\end{align} 
\end{proposition}

\begin{proof}
The estimates \eqref{e:vq:vell}--\eqref{e:Rq:Dt} are consequence of Propositions \ref{p:vi:vell} and \ref{e:ziBidef} (the proof can be found in \cite{BDLSV2017}). However we prove explicitly  \eqref{e:voverline_vell_energy_diff} since it involves the structure of the dissipative term.

Observe that for $t\in I_i$
\begin{align*}
 \overline v_q \otimes \overline v_q&=(\chi_iv_i+(1-\chi_i)v_{i+1})\otimes(\chi_iv_i+(1-\chi_i)v_{i+1})\\
 &=\chi_iv_i\otimes v_i+(1-\chi_i)v_{i+1}\otimes v_{i+1}-\chi_i(1-\chi_i)(v_i-v_{i+1})\otimes (v_i-v_{i+1}),
 \end{align*}
 so that, taking the trace:
 \begin{align*}
|\overline v_q|^2-|v_\ell|^2=\chi_i(|v_i|^2-|v_\ell|^2)+(1-\chi_i)(|v_{i+1}|^2-|v_\ell|^2)-\chi_i(1-\chi_i)|v_i-v_{i+1}|^2
\end{align*}
Next, recall that $v_i$ and $v_\ell$ are smooth solutions of \eqref{e:vi:def} and \eqref{e:euler_reynolds_l} respectively, therefore
\begin{align*}
\left|\frac{d}{dt}\int_{\T^3}|v_i|^2-|v_\ell|^2\,dx\right|=2 \left|\int_{\T^3}\nabla v_\ell:\mathring{R}_\ell\,dx\right| +2 \nu \left|\int_{\T^3}\bigr(|(-\Delta)^{\sfrac{\gamma}{2}}v_i|^2-|(-\Delta)^{\sfrac{\gamma}{2}}v_\ell|^2\bigr)\,dx\right|\, .
\end{align*}
Using \eqref{e:R:ell} and \eqref{e:z:N+alpha}, we estimate
\[
\left|\int_{\T^3}\nabla v_\ell:\mathring{R}_\ell\,dx\right|\lesssim \|\nabla v_\ell\|_0 \|\mathring{R}_\ell\|_0\lesssim \delta_q^{\sfrac{1}{2}} \lambda_q \delta_{q+1} \lesssim\tau_q^{-1}\delta_{q+1}\ell^\alpha\, .
\]
Moreover, since $\|v_q\|_\gamma \leq 1$ for every $\gamma <\beta$ (as already exploited in the proof of Proposition \ref{p:main}), by \eqref{e:z_diff_k}, Theorem \ref{lapla.holder} and Cauchy-Schwarz inequality we have
\[
\left|\int_{\T^3}\bigr(|(-\Delta)^{\sfrac{\gamma}{2}}v_i|^2-|(-\Delta)^{\sfrac{\gamma}{2}}v_\ell|^2\bigr)\,dx\right|\lesssim \|v_i-v_\ell\|_{\gamma+\alpha}\lesssim \tau_q \delta_{q+1} \ell^{-1-\gamma}\lesssim \tau_q^{-1} \delta_{q+1} \ell^{\alpha}\, ,
\]
where in the last inequality (remember the restriction $\gamma <\sfrac{1}{3}$) we have used 
\[
 \ell^{-1-\gamma}\leq \ell^{-\sfrac{4}{3}}\overset{\eqref{e:ell_def}}{=}\frac{\big(\delta_q^{\sfrac{1}{2}}\lambda_q\big)^{\sfrac{4}{3}}}{\delta_{q+1}^{\sfrac{2}{3}}}\lambda_q^{2\alpha} \overset{\eqref{e:tau_def}}{=}\tau_q^{-2} \ell^\alpha \ell^{\sfrac{5\alpha}{3}}\lambda_q^{2\alpha} \big(\tau_q \delta_{q+1}^{-1}\big)^{\sfrac{2}{3}}\overset{\eqref{e:compare-lambda-ell} }{\leq}\tau_q^{-2} \ell^\alpha\,.
 \]
Moreover, $v_i=v_\ell$ for $t=t_i$. Therefore, after integrating in time we deduce
$$
\left|\int_{\T^3}|v_i|^2-|v_\ell|^2\,dx\right|\lesssim \delta_{q+1}\ell^\alpha.
$$
Furthermore, using \eqref{e:z_diff_k} and $\delta_{q+1}^{\sfrac{1}{2}} \tau_q \ell^{-1}=\ell^{2\alpha} \lambda_q^{\sfrac{3\alpha}{2}} \overset{\eqref{e:compare-lambda-ell} }{\leq} \lambda_q^{-\sfrac{\alpha}{2}}\leq 1$
$$
\int_{\T^3}|v_i-v_{i+1}|^2\,dx\lesssim \|v_i-v_{i+1}\|_\alpha^2\lesssim \tau_q^2\delta_{q+1}^2\ell^{-2+2\alpha} \lesssim \delta_{q+1}\ell^{2\alpha},
$$
Therefore 
$$
\left| \int |\bar v_q|^2 - |v_\ell |^2 dx \right|\lesssim \delta_{q+1}\ell^{\alpha},
$$
which concludes the proof.

\end{proof}

\subsection{Perturbation and Mikado flows}\label{s:perturbation_outline}

We will now outline the construction of the perturbation $w_{q+1}$, where 
\[
v_{q+1}:= w_{q+1} + \overline v_q \, .
\] 
The perturbation $w_{q+1}$ is highly oscillatory and will be based on the Mikado flows introduced in \cite{DaSz2016}.

First of all note that as a corollary of \eqref{e:energy_inductive_assumption}, \eqref{e:vq_vell_energy_diff} and \eqref{e:voverline_vell_energy_diff}, by choosing $a$ sufficiently large we can ensure that
\begin{equation}\label{e:voverline_energy_error}
\frac{\delta_{q+1}}{2\lambda_q^{\alpha}}\le e(t)-\int_{\T^3}\abs{\overline v_q}^2\,dx\leq 2\delta_{q+1}\,.
\end{equation}
Starting with the solution $(\overline{v}_q,\overline{p}_q,\mathring{\overline R_q})$, we then produce a new solution $(v_{q+1},p_{q+1},\mathring{R}_{q+1})$ of the Navier-Stokes Reynolds system \eqref{NSR} with estimates
\begin{align}
\|v_{q+1}-\overline v_q\|_0 + \lambda_{q+1}^{-1} \|v_{q+1}-\overline v_q\|_1&\leq \frac{M}{2}\delta_{q+1}^{\sfrac12}\label{e:outline_v_diff}\\
\|\mathring{R}_{q+1}\|_\alpha &\lesssim \frac{\delta_{q+1}^{\sfrac12}\delta_q^{\sfrac12}\lambda_q}{\lambda_{q+1}^{1-4\alpha}}\,.\label{e:outline_R_est}\\
    \abs{e(t)-\int_{\T^3}\abs{v_{q+1}}^2\,dx-\frac{\delta_{q+2}}2 } &\lesssim \frac{\delta_q^{\sfrac12}\delta_{q+1}^{\sfrac12}\lambda_q^{1+ 2\alpha}}{\lambda_{q+1}}\, ,\label{e:outline_energy_diff}
\end{align}
cf. Propositions \ref{p:perturbation}, \ref{p:energy} and \ref{p:R_q+1}.

Then Proposition \ref{p:main} is just a consequence of estimates \eqref{e:outline_v_diff}-\eqref{e:outline_energy_diff}, Proposition \ref{p:est_mollification} and Proposition \ref{p:vq:vell} (again, a detailed proof can be found in \cite{BDLSV2017}).

We now recall the construction of Mikado flows given in \cite{DaSz2016}.

\begin{lemma}\label{l:Mikado}For any compact subset $\mathcal N\subset\subset \S^{3\times3}_+$
there exists a smooth vector field 
$$
W:\mathcal N\times \T^3 \to \R^3, 
$$
such that, for every $R\in\mathcal N$ 
\begin{equation}\label{e:Mikado}
\left\{\begin{aligned}
\div_\xi(W(R,\xi)\otimes W(R,\xi))&=0 \\ \\
\div_\xi W(R,\xi)&=0,
\end{aligned}\right.
\end{equation}
and
\begin{eqnarray}
	\fint_{\T^3} W(R,\xi)\,d\xi&=&0,\label{e:MikadoW}\\
    \fint_{\T^3} W(R,\xi)\otimes W(R,\xi)\,d\xi&=&R.\label{e:MikadoWW}
\end{eqnarray}
\end{lemma}

Using the fact that $W(R,\xi)$ is $\T^3$-periodic and has zero mean in $\xi$, we write
\begin{equation}\label{e:Mikado_Fourier}
W(R,\xi)=\sum_{k \in \Z^3\setminus\{0\}}  a_k(R) e^{ik\cdot \xi}
\end{equation}
for some  smooth functions $R\to a_k(R) \in \C^3$, satisfying $a_k (R) \cdot k=0$. From the smoothness of $W$, we further infer
\begin{equation}\label{e:a_k_est}
\sup_{R\in \mathcal N}\abs{D^N_R a_k(R)}\leq  \frac{C(\mathcal{N},N,m)}{\abs {k}^m}
\end{equation}
for some constant $C$, which depends, as highlighted in the statement, on $\mathcal{N}$, $N$ and $m$.
\begin{remark}\label{r:choice_of_M}
Later in the proof the estimates \eqref{e:a_k_est} will be used with a specific choice of the compact set $\mathcal{N}$ and of the integers $N$ and $m$: this specific choice will then determine the universal constant $M$ appearing in Proposition \ref{p:main}.
\end{remark}

Using the Fourier representation we see that from \eqref{e:MikadoWW}
\begin{equation}\label{e:Mikado_stationarity}
W(R,\xi)\otimes W(R,\xi) = R+\sum_{k\neq 0} C_{k}(R) e^{i k\cdot \xi}
\end{equation}
where
\begin{equation}\label{e:Ck_ind}
C_k  k=0 \quad \mbox{and} \quad
\sup_{R\in \mathcal N}\abs{D^N_R C_k(R)}\leq \frac{C (\mathcal{N}, N, m)}{\abs {k}^m}
\end{equation}
for any $m,N \in \N$. 

It will also be useful to write the Mikado flows in terms of a potential. We note
\begin{align}
\curl_{\xi}\left(\left(\frac{ik\times  a_k}{\abs{k}^2}\right) e^{i k\cdot \xi}\right) &= -i\left(\frac{ik\times  a_k}{\abs{k}^2}\right)\times k  e^{i k\cdot \xi} 
= -\frac{k\times (k\times  a_k)}{\abs{k}^2}  e^{i k\cdot \xi} =  a_k  e^{i k\cdot \xi} \label{e:Mikado_Potential}
\end{align}

Recall that $\mathring{\overline R_q}$ is supported in the set $\T^3\times \bigcup_iI_i$, whereas, from \eqref{e:chi:time:width} it follows that 
$[0,T]\setminus  \bigcup_iI_i=\bigcup_iJ_i$, where the open intervals $J_i$ have length $|J_i|=\tfrac 23\tau_q$ each, except for the first $J_0$ and last one, which might be shortened by the intersection with $[0,T]$, more precisely
$$
J_i=(t_{i+1}- \tfrac 13 \tau_q,t_{i+1}+\tfrac 13\tau_q) \cap [0,T] \, .
$$
We start by defining smooth non-negative cut-off functions $\eta_i=\eta_i(x,t)$ with the following properties
\begin{enumerate}
\item[(i)] $\eta_i\in C^{\infty}(\T^3\times [0,T])$ with $0\leq \eta_i(x,t)\leq 1$ for all $(x,t)$;
\item[(ii)] $\supp \eta_i\cap\supp\eta_j=\emptyset$ for $i\neq j$;
\item[(iii)] $\T^3\times I_i\subset \{(x,t):\eta_i(x,t)=1\}$;
\item[(iv)] $\supp \eta_i\subset \T^3\times I_i\cup J_i\cup J_{i+1}$;
\item[(v)] There exists a positive geometric constant $c_0>0$ such that for any $t\in[0,T]$
$$
\sum_i\int_{\T^3}\eta_i^2(x,t)\,dx\geq c_0.
$$
\end{enumerate}

The next lemma is taken from \cite{BDLSV2017}.
\begin{lemma}\label{l:cutoffs}
There exists cut-off functions $\{\eta_i\}_i$ with the properties (i)-(v) above and such that for any $i$ and $n,m\geq 0$
\begin{align*}
\|\partial_t^n\eta_i\|_{m}\leq C (n,m) \tau_q^{-n}
\end{align*}
where $C(n,m)$ are geometric constants depending only upon $m$ and $n$.
\end{lemma}

Define 
\begin{equation*}
\rho_{q}(t):= \frac{1}{3} \left(e(t)-\frac{\delta_{q+2}}{2}-\int_{\T^3}\abs{\overline v_q}^2\,dx\right)
\end{equation*}
and
\begin{equation*}
\rho_{q,i}(x,t):= \frac{\eta_i^2(x,t)}{\sum_j \int_{\T^3} \eta_j^2(y,t)\,dy}\rho_{q}(t)
\end{equation*}

Define the backward flows $\Phi_i$ for the velocity field $\overline v_{q}$ as the solution of the transport equation
\begin{equation*}
\left\{ 
\begin{aligned}
&(\partial_t + \overline v_q  \cdot \nabla) \Phi_i =0 \\ \\
&\Phi_i\left(x,t_i\right) = x.
\end{aligned}
\right.
\end{equation*}
Define
\begin{equation*}
R_{q,i}:=\rho_{q,i} \Id- \eta_i^2\mathring{\overline R_q}
\end{equation*}
and
\begin{equation}\label{e:tildeR_def}
\tilde R_{q,i} =  \frac{\nabla\Phi_iR_{q,i}(\nabla\Phi_i)^T}{ \rho_{q,i}} \,.
\end{equation}
We note that, because of properties (ii)-(iv) of $\eta_i$, 
\begin{itemize}
\item $\supp R_{q,i}\subset \supp\eta_i$;
\item on $\supp\mathring{\bar{R}}_q$ we have $\sum_i\eta_i^2=1$;
\item $\supp \tilde R_{q,i}\subset \T^3\times I_i\cup J_i\cup J_{i+1}$;
\item $\supp \tilde R_{q,i}\cap \supp \tilde R_{q,j}=\emptyset\textrm{ for all }i\neq j$.
\end{itemize}

\begin{lemma}
\label{l:R_in_range}
For $a\gg 1$ sufficiently large we have
\begin{equation}\label{e:Phi-close-to-id}
\|\nabla \Phi_i - \Id\|_0 \leq \frac{1}{2} \qquad \mbox{for $t\in \supp (\eta_i)$.}
\end{equation}
Furthermore, for any $N\geq 0$ 
\begin{align}
\frac{\delta_{q+1}}{8\lambda_q^{\alpha}} \leq |\rho_{q}(t)| &\leq \delta_{q+1}\quad\textrm{ for all $t$}\,,
\label{e:rho_range}\\
\norm{\rho_{q,i}}_0 &\leq \frac{\delta_{q+1}}{c_0}\,,\label{e:rho_i_bnd}\\
 \norm{\rho_{q,i}}_N&\lesssim \delta_{q+1}\,,\label{e:rho_i_bnd_N}\\
 \norm{\partial_t \rho_q}_0 &\lesssim \delta_{q+1} \delta_q^{\sfrac{1}{2}} \lambda_q
 \label{e:rho_t}\,,\\
 \norm{\partial_t \rho_{q,i}}_N &\lesssim \delta_{q+1}\tau_q^{-1}\,.
 \label{e:rho_i_bnd_t}
\end{align}
Moreover, for all $(x,t)$
$$
\tilde R_{q,i}(x,t)\in B_{\sfrac12}(\Id)\subset \mathcal{S}^{3\times 3}_+\,,
$$
where $B_{\sfrac12}(\Id)$ denotes the metric ball of radius $1/2$ around the identity $\Id$ in the space $\mathcal{S}^{3\times 3}$. 
\end{lemma}
\begin{proof}[Proof of Lemma~\ref{l:R_in_range}]
For the estimates \eqref{e:rho_range}-\eqref{e:rho_i_bnd_N} we refer to \cite{BDLSV2017}.
 Note that by the definition of the cut-off functions $\eta_i$ 
\begin{equation}\label{e:sum_eta_range}
c_0 \leq \sum_i \int_{\T^3} \eta_i^2(y,t)\,dy \leq 2\,.
\end{equation}
To prove \eqref{e:rho_t} and \eqref{e:rho_i_bnd_t} we first note that
\begin{equation*}
\abs{\frac{d}{dt} \int \abs{\overline v_{q}(x,t)}^2\,dx}\leq  2\abs{\int \nabla \overline v_q\cdot \mathring{\overline R_q}\,dx }+2\nu \int |(-\Delta)^{\sfrac{\gamma}{2}}\overline{v}_q|^2 \,dx \lesssim \delta_{q+1}\delta_q^{\sfrac 12}\lambda_q \ell^\alpha
\end{equation*}
Thus\footnote{Note that $\|\partial_t e\|_0 \leq 1 \leq \delta_{q+1} \delta_q^{\frac{1}{2}} \lambda_q$ since $\delta_{q+1} \delta_q^{\sfrac{1}{2}} \lambda_q = \lambda_{q+1}^{-2\beta} \lambda_q^{1-\beta} 
\geq  a^{b^q(1-\beta - 2\beta b)} \geq 1$ (recall that $b < \frac{1-\beta}{2\beta}$).} 
\[
 \norm{\partial_t \rho_{q}}_0\lesssim \delta_{q+1}\delta_q^{\sfrac 12}\lambda_q 
\]
Then, since $\|\partial_t \eta_j\|_N \lesssim \tau_q^{-1}$ and $\delta_q^{\sfrac 12}\lambda_q\leq \tau_q^{-1}$, 
using \eqref{e:sum_eta_range}, the estimate \eqref{e:rho_i_bnd_t} follows.
\end{proof}

\subsection{The constant \texorpdfstring{$M$}{M}}
The principal term of the perturbation can be written as
\begin{equation}
w_{o}:=\sum_i  \left(\rho_{q,i}(x,t)\right)^{\sfrac12} (\nabla\Phi_i)^{-1} W(\tilde R_{q,i}, \lambda_{q+1}\Phi_i) = \sum_i w_{o,i}\, ,
\label{e:w0_decomp}
\end{equation}
where Lemma \ref{l:Mikado} is applied with $\mathcal{N} = \overline{B}_{\sfrac12} (\Id)$, namely the closed ball (in the space of symmetric $3\times 3$ matrices) of radius $\sfrac{1}{2}$ centered at the identity matrix.

From Lemma \ref{l:R_in_range} it follows that $W(\tilde R_{q,i}, \lambda_{q+1}\Phi_i)$ is well defined. Using the Fourier series representation of the Mikado flows \eqref{e:Mikado_Fourier} we can write 
\begin{equation*}
w_{o,i}=\sum_{k \neq 0} (\nabla\Phi_i)^{-1} b_{i,k} e^{i\lambda_{q+1}k\cdot \Phi_i}\, ,
\end{equation*}
where 
\begin{equation*}
b_{i,k}(x,t):= \left(\rho_{q,i}(x,t)\right)^{\sfrac12} a_k(\tilde R_{q,i}(x,t)).
\end{equation*}
The following is a crucial point of our construction, which ensures that the constant $M$ of Proposition \ref{p:main}
is geometric and in particular independent of all the parameters of the construction.

\begin{lemma}\label{l:choice_of_M}
There is a geometric constant $\bar M$ such that
\begin{equation}\label{e:barM}
\|b_{i,k}\|_0 \leq \frac{\bar M}{|k|^{ 4}} \delta_{q+1}^{\sfrac{1}{2}}\, .
\end{equation}
\end{lemma}

We are finally ready to define the constant $M$ of Proposition \ref{p:main}: from Lemma \ref{l:choice_of_M} it follows trivially that the constant is indeed geometric and hence independent of all the parameters entering in the statement of Proposition \ref{p:main}.

We can now define the geometric constant $M$ as 
\begin{equation}\label{d:choice_of_M}
M = 64 \bar M \sum_{k\in \Z^3\setminus \{0\}} \frac{1}{|k|^4}\, ,
\end{equation}
where $\bar M$ is the constant of Lemma \ref{l:choice_of_M}.

We also define
\begin{align*}
w_{c}&:=\frac{-i}{\lambda_{q+1}}\sum_{i,k\neq 0} \left[ \curl \left(\left(\rho_{q,i}\right)^{\sfrac12} \frac{\nabla\Phi_i^T(k\times a_{k}(\tilde R_{q,i}))}{\abs{k}^2}\right)\right] e^{i\lambda_{q+1}k\cdot \Phi_i}
=: \sum_{i,k\neq 0} c_{i,k} e^{i\lambda_{q+1}k\cdot \Phi_i}\, .
\end{align*}
Then by direct computations one can check that
\begin{align}\label{e:w_curl}
w_{q+1} = w_o+w_c=\frac{-1}{\lambda_{q+1}}\curl\left(\sum_{i,k\neq 0} (\nabla\Phi_i)^T\left(\frac{ik\times b_{k,i}}{\abs{k}^2}\right)e^{i\lambda_{q+1}k\cdot \Phi_i}\right)\,,
\end{align}
thus the perturbation $w_{q+1}$ is divergence free. Note that the dependence of $w_{q+1}(\cdot,0)$ on the function $e(t)$ is only trough the value $e(0)$.

\subsection{The final Reynolds stress and conclusions}  Upon letting
\begin{align*}
\overline R_q = \sum_{i} R_{q,i}\, ,
\end{align*}
the new Reynolds stress will be split in two main component: the Euler error $\mathring{R}_{q+1}^E$ and the dissipative error $\mathring{R}_{q+1}^D$, i.e. 
\begin{align}
\mathring{R}_{q+1} = \mathring{R}_{q+1}^E+\mathring{R}_{q+1}^D \, ,
 \label{e:R:q+1} 
\end{align}
where 
\begin{align}
& \mathring{R}_{q+1}^E :=  \RR \left( w_{q+1} \cdot \nabla \overline v_q+\partial_t  w_{q+1} + \overline v_q \cdot \nabla w_{q+1}+\div \left(- {\overline R}_{q} + (w_{q+1} \otimes w_{q+1}) \right)\right) \\
& \mathring{R}_{q+1}^D:=\nu \RR \left( \gammalaplace w_{q+1} \right).
\end{align}

Notice that all three terms in \eqref{e:R:q+1} are of the form $\RR f$, where $f$ has always zero mean. Notice also that the definition of $ \mathring{R}_{q+1}^E$ is the same as in \cite{BDLSV2017} and that due to the dissipative term $\gammalaplace$ we have to put also $ \mathring{R}_{q+1}^D$ in the definition of the new Reynolds stress in order to ensure that the system \eqref{NSR} is satisfied at the step $q+1$.
Indeed, with this definition one may verify that 
\begin{equation*}
\left\{
\begin{array}{l}
 \partial_t v_{q+1} + \div (v_{q+1} \otimes v_{q+1}) + \nabla p_{q+1} +\nu \gammalaplace v_{q+1}= \div(\mathring{R}_{q+1}) \, ,
\\ \\
 \div v_{q+1} = 0 \, ,
\end{array}\right.
\end{equation*}
where the new pressure is defined by
\begin{equation}\label{e:new_pressure}
p_{q+1}(x,t) = \bar p_q(x,t)  - \sum_{i} \rho_{q,i}(x,t)  + \rho_{q}(t).
\end{equation}

We now state a proposition taken from \cite{BDLSV2017}.
\begin{proposition}
\label{p:perturbation}
For $t\in \tilde I_i$ and any $N\geq 0$
\begin{align}
\norm{ (\nabla\Phi_i)^{-1}}_N + \norm{\nabla\Phi_i}_N &\lesssim \ell^{-N} \,,\label{e:phi_N}\\
\norm{\tilde R_{q,i}}_N &\lesssim  \ell^{-N}\,,\label{e:tR_est}\\
\norm{b_{i,k}}_N &\lesssim \delta_{q+1}^{\sfrac12}|k|^{-6}\ell^{-N}\,, \label{e:b_k_est_N}\\
\norm{c_{i,k}}_N &\lesssim  \delta_{q+1}^{\sfrac12}\lambda_{q+1}^{-1}|k|^{-6}\ell^{-N-1}\,.\label{e:c_k_est}
\end{align}
Moreover assuming $a$ is sufficiently large, the perturbations $w_o$, $w_c$ and $w_q$ satisfy the following estimates
\begin{align}
\norm{w_o}_0 +\frac{1}{\lambda_{q+1}}\norm{w_o}_1 &\leq \frac{M}{4}\delta_{q+1}^{\sfrac 12}\label{e:w_o_est}\\
\norm{w_c}_0+\frac{1}{\lambda_{q+1}} \norm{w_c}_1 &\lesssim \delta_{q+1}^{\sfrac 12}\ell^{-1}\lambda_{q+1}^{-1}\label{e:w_c_est}\\
\norm{w_{q+1}}_0 +\frac{1}{\lambda_{q+1}}\norm{w_{q+1}}_1 &\leq  \frac{M}{2} \delta_{q+1}^{\sfrac 12}\label{e:w_est}
\end{align}
where the constant $M$ depends solely on the constant $c_0$ in \eqref{e:sum_eta_range}.
In particular, we obtain \eqref{e:outline_v_diff}.
\end{proposition}

%%%%%%%%%%%%%%%%%%%%%%%%%%%%%%%%%%%%%%%%%%%%%%%%%%%%%%%%%%%%%
%%%%%%%%%%%%%%%%%%%%%%%%%%%%%%%%%%%%%%%%%%%%%%%%%%%%%%%%%%%%%
%%%%%%%%%%%%%%%%%%%%%%%%%%%%%%%%%%%%%%%%%%%%%%%%%%%%%%%%%%%%%

We are now ready to complete the proof of Proposition \ref{p:main} by proving the remaining estimates \eqref{e:outline_R_est} and \eqref{e:outline_energy_diff}. The estimate \eqref{e:outline_energy_diff} is a consequence of Proposition \ref{p:perturbation} and Lemma \ref{l:R_in_range} and does not involve the different structure of the Navier-Stokes equations with respect to the Euler ones, thus for the proof of the next proposition we refer to \cite{BDLSV2017}.
\begin{proposition}\label{p:energy}
 The energy of $v_{q+1}$ satisfies the following estimate:
\begin{equation*}
    \abs{e(t)-\int_{\T^3}\abs{v_{q+1}}^2\,dx-\frac{\delta_{q+2}}2 }\lesssim \frac{\delta_q^{\sfrac12}\delta_{q+1}^{\sfrac12}\lambda_q^{1+2\alpha}}{\lambda_{q+1}}\,.
\end{equation*}
\end{proposition}

For the inductive estimate on $\mathring R_{q+1}$ we have the following
\begin{proposition}
\label{p:R_q+1}
The Reynolds stress error $\mathring R_{q+1}$ defined in \eqref{e:R:q+1} satisfies the estimate
\begin{equation}\label{e:final_R_est}
\norm{\mathring R_{q+1}}_{0}\lesssim \frac{\delta_{q+1}^{\sfrac12}\delta_q^{\sfrac{1}{2}} \lambda_q}{\lambda_{q+1}^{1-4\alpha}} \,.
\end{equation}
In particular, \eqref{e:outline_R_est} holds.
\end{proposition}
\begin{proof}
For the first term in the definition of the new Reynolds stress tensor we have
\[
\norm{\mathring R_{q+1}^E}_{0}\lesssim \frac{\delta_{q+1}^{\sfrac12}\delta_q^{\sfrac{1}{2}} \lambda_q}{\lambda_{q+1}^{1-4\alpha}} \,.
\]
We are not going to give the proof of the last estimate because, as already explained, it can be found in \cite[Proposition 6.1]{BDLSV2017}.
To estimate $\mathring R_{q+1}^D $ we first note that $\nu<1$ and the two operators $\RR$ and $\gammalaplace$ commute, therefore we can first estimate $ \norm{ \RR w_{q+1}}_0$ and $ \norm{ \RR w_{q+1}}_1$ from which, using Theorem \ref{lapla.holder} and interpolation in H\"older spaces, we conclude
\begin{equation}
\norm{\mathring R_{q+1}^D}_{0}\lesssim \norm{ \RR w_{q+1}}_{2\gamma+\alpha} \lesssim \norm{ \RR w_{q+1}}_0^{1-2\gamma-\alpha} \norm{ \RR w_{q+1}}_1^{2\gamma+\alpha}\, .
\end{equation}
By the definition of the new perturbations we have
\begin{align*}
& w_c=\sum_{i,k\neq 0} c_{i,k} e^{i\lambda_{q+1}k\cdot \Phi_i} \\
& w_o=\sum_{i,k\neq 0} L_{i,k} e^{i\lambda_{q+1}k\cdot \Phi_i}\, ,
\end{align*}
where $ L_{i,k}:=(\nabla\Phi_i)^{-1} b_{i,k}$. Using Proposition \ref{p:perturbation} we have
\begin{equation}\label{Likest}
\|L_{i,k}\|_N\leq \| (\nabla\Phi_i)^{-1}\|_N \| b_{i,k}\|_0+ \| (\nabla\Phi_i)^{-1}\|_0 \| b_{i,k}\|_N \lesssim \delta_{q+1}^{\sfrac12}|k|^{-6}\ell^{-N}\, .
\end{equation}
Using Proposition \ref{p:R:oscillatory_phase} and \eqref{Likest} we estimate
\begin{align*}
 \| \RR w_o\|_0 \leq \| \RR w_o\|_\alpha &\lesssim \sum_{i,k\neq 0} \frac{  \norm{L_{i,k}}_{0}}{\lambda_{q+1}^{1-\alpha}|k|^{1-\alpha}} +   \frac{ \norm{L_{i,k}}_{N+\alpha}+\norm{L_{i,k}}_{0}\norm{\Phi_i}_{N+\alpha}}{\lambda_{q+1}^{N-\alpha}|k|^{N-\alpha}}  \\
& \lesssim \delta_{q+1}^{\sfrac12} \sum_{k\neq 0} \frac{ 1 }{\lambda_{q+1}^{1-\alpha}|k|^{7-\alpha}} +   \frac{ \ell^{-N-\alpha}}{\lambda_{q+1}^{N-\alpha}|k|^{N-\alpha+7}} \lesssim \frac{\delta_{q+1}^{\sfrac12}}{\lambda_{q+1}^{1-\alpha}}\, ,
\end{align*}
where in the last inequality we have chosen $N$ big enough. It is not difficoult to see that we also have
\[
\| \RR w_o\|_1 \lesssim \delta_{q+1}^{\sfrac12}\lambda_{q+1}^{\alpha}\, ,
\]
since each term where we take a first derivative (in space) will gain a factor $\lambda_{q+1}$, thus by interpolation we conclude
\begin{equation}\label{stimawo}
\| \gammalaplace \RR w_o \|_0 \lesssim  \frac{\delta_{q+1}^{\sfrac12}}{\lambda_{q+1}^{1-\alpha}} \lambda_{q+1}^{2\gamma}\, .
\end{equation}
Now we observe that the estimate on the coefficients $c_{i,k}$ are better then those for the $L_{i,k}$'s, so that we also bound
\begin{equation}\label{stimawc}
\| \gammalaplace \RR w_c \|_0 \lesssim  \frac{\delta_{q+1}^{\sfrac12}}{\lambda_{q+1}^{1-\alpha}} \lambda_{q+1}^{2\gamma}\, .
\end{equation}
Finally, combining \eqref{stimawo}, \eqref{stimawc} and the restriction $\gamma <\sfrac{1}{3}$ we get
\[
\norm{\mathring R_{q+1}^D}_{0}\lesssim \frac{\delta_{q+1}^{\sfrac12}}{\lambda_{q+1}^{1-\alpha}} \lambda_{q+1}^{2\gamma} \lesssim \frac{\delta_{q+1}^{\sfrac12}\delta_q^{\sfrac{1}{2}} \lambda_q}{\lambda_{q+1}^{1-4\alpha}} \,.
\]
\end{proof}

\appendix

%%%%%%%%%%%%%%%%%%%%%%%%%
\section{Some estimates in H\"older spaces}\label{s:hoelder}
%%%%%%%%%%%%%%%%%%%%%%%%%

Recall the following elementary inequalities
\begin{proposition}
Let $f, g$ be two smooth functions. For any $r\geq s\geq 0$ we have
\begin{align}
[fg]_{r}&\leq C\bigl([f]_r\|g\|_0+\|f\|_0[g]_r\bigr)\label{e:Holderproduct}\\
[f]_{s}&\leq C\|f\|_0^{1-\sfrac{s}{r}}[f]_{r}^{\sfrac{s}{r}}.\label{e:Holderinterpolation2}
\end{align}
\end{proposition}

We also recall the quadratic commutator estimate of~\cite{CoETi1994} 
(cf. also \cite[Lemma 1]{CoDLSz2012}):

\begin{proposition}\label{p:CET}
Let $f,g\in C^{\infty}(\T^3\times\T)$ and $\psi$ a standard radial smooth and compactly supported kernel. For any $r\geq 0$  we have the estimate
\[
\Bigl\|(f*\psi_\ell)( g*\psi_\ell)-(fg)*\psi_\ell\Bigr\|_r\leq C\ell^{2-r}  \|f\|_1\|g\|_1 \, ,
\]
where the constant $C$ depends only on $r$.
\end{proposition}

\section{Estimates on the fractional Laplacian}

\begin{theorem}{(Interaction with Holder spaces)}\label{lapla.holder}
Let $\gamma, \varepsilon>0$ and $\beta\geq0$  such that $2\gamma+\beta+\varepsilon\leq 1$, and let $f:\mathbb{T}^3 \rightarrow \mathbb{R}^3$. If  $f \in C^{0,2 \gamma +\beta+ \varepsilon} $,  then $\gammalaplace f \in C^\beta$, moreover there exists a constant $C=C(\varepsilon)>0$ such that 
\begin{equation}\label{roncalholderformula}
\| (-\Delta)^\gamma f\|_\beta \leq C(\varepsilon) [f]_{2 \gamma +\beta+ \varepsilon}.
\end{equation}

 \end{theorem}
\begin{proof}
The proof of \eqref{roncalholderformula} for $\beta=0$ can be found in \cite{RS2016}, Theorem 1.4. Fix now $\beta>0$. For any $ h \in \T^3$ we have
\begin{align*}
\| \gammalaplace (f(\cdot+h)-f(\cdot)) \|_0\leq C(\varepsilon)[ f(\cdot+h)-f(\cdot)]_{2\gamma+\varepsilon}\leq C(\varepsilon)|h|^\beta [f]_{2\gamma+\beta+\varepsilon}\, ,
\end{align*}
from which \eqref{roncalholderformula} follows.
\end{proof}

\begin{corollary}\label{vlaplv}
Let $\gamma \in (0,1)$, $\varepsilon >0$ be such that $0<\gamma+\varepsilon\leq 1$, and let $f:\mathbb{T}^3 \rightarrow \mathbb{R}^3$. There exist a constant $C=C(\varepsilon)>0$ such that
\begin{equation}
\int_{\mathbb{T}^3} |(-\Delta)^{\sfrac{\gamma}{2}} f|^2 (x) dx \leq C(\varepsilon) [f]_{\gamma+\varepsilon}^2
\qquad \forall f \in C^{\gamma+\varepsilon} (\T^3)\, .
\end{equation}
\end{corollary}

%%%%%%%%%%%%%%%%%%%%%%%%%
\section{Potential theory estimates}
%%%%%%%%%%%%%%%%%%%%%%%%%
We recall the definition of the standard class of periodic Calder{\'o}n-Zygmund operators.
Let $K$ be an $\R^3$ kernel which obeys the properties
\begin{itemize}
\item $ K(z) = \Omega\left(\frac{z}{|z|}\right) |z|^{-3} $, for all $z\in\R^3 \setminus \{0\}$ 
\item $\Omega \in C^\infty({\mathbb S}^2)$
\item $\int_{|\hat z|=1} \Omega(\hat z) d\hat z = 0$.
\end{itemize}
From the $\R^3$ kernel $K$, use Poisson summation to define the periodic kernel 
\begin{align*}
K_{\T^3}(z) =  K(z) + \sum_{\ell \in {\mathbb Z}^3 \setminus \{0\}} \left( K(z+\ell) - K(\ell) \right).
\end{align*}
Then the operator
\begin{align*}
T_K f(x) = p.v. \int_{\T^3} K_{\T^3}(x-y) f(y) dy
\end{align*}
is a $\T^3$-periodic Calder{\'o}n-Zygmund operator, acting on $\T^3$-periodic functions $f$ with zero mean on $\T^3$.
The following proposition, proving the boundedness of periodic Calder{\'o}n-Zygmund operators on periodic H\"older spaces is classical (see \cite{CZ1954}).
\begin{proposition}
\label{p:CZO_C_alpha}
Fix $\alpha \in (0,1)$. Periodic Calder{\'o}n-Zygmund operators are bounded on the space of zero mean $\T^3$-periodic $C^\alpha$ functions. 
\end{proposition}
 
The following is a simple consequence of classical stationary phase techniques. For a detailed proof the reader might consult \cite[Lemma 2.2]{DaSz2016}.
\begin{proposition}
\label{p:R:oscillatory_phase}
Let  $\alpha \in(0,1)$ and $N \geq 1$. Let $a \in C^\infty(\T^3)$, $\Phi\in C^\infty(\T^3;\R^3)$ be smooth functions and assume that
\begin{equation*}
\hat{C}^{-1}\leq \abs{\nabla \Phi}, | \nabla \Phi^{-1}|\leq \hat{C}
\end{equation*}
holds on $\T^3$. Then 
\begin{equation}\label{e:phase_integration}
\abs{\int_{\T^3}a(x)e^{ik\cdot \Phi}\,dx}\lesssim  \frac{ \norm{a}_{N}+\norm{a}_{0}\norm{\Phi}_{N}}{|k|^{N}} \, ,
\end{equation}
and for the operator $\RR$ defined in \eqref{e:R:def}, we have
 \begin{align*}
\norm{ {\mathcal R} \left(a(x) e^{i k\cdot \Phi} \right)}_{\alpha} 
&\lesssim \frac{  \norm{a}_{0}}{|k|^{1-\alpha}} +   \frac{ \norm{a}_{N+\alpha}+\norm{a}_{0}\norm{\Phi}_{N+\alpha}}{|k|^{N-\alpha}} \, ,
\end{align*}
where the implicit constant depends on $\hat{C}$, $\alpha$ and $N$, but not on $k$.
\end{proposition}

%%%%%%%%%%%%% BIBLIOGRAPHY %%%%%%%%%%

%\bibliographystyle{alpha}
%\bibliography{Onsager}

\end{document}